\documentclass{article}
\usepackage[utf8]{inputenc}

\usepackage{cite}
\usepackage{graphicx}
\usepackage{graphicx}
\usepackage{pstricks}
\usepackage{appendix}
\usepackage{dsfont}
\usepackage{amsmath}
\usepackage{amssymb}
\usepackage{amsthm}
\usepackage{mathtools}
\usepackage{comment}
\usepackage[shortlabels]{enumitem}
\usepackage{tikz,lipsum,lmodern}
\usepackage[most]{tcolorbox}

\usepackage{xcolor}
\usepackage{blindtext}

\usepackage[nottoc]{tocbibind}

\numberwithin{equation}{section}
\newtheorem{theorem}{Theorem}[section]
\newtheorem{corollary}[theorem]{Corollary}
\newtheorem{definition}[theorem]{Definition}
\newtheorem{assumption}[theorem]{Assumption}
\newtheorem{proposition}[theorem]{Proposition}

\newtheorem{lemma}[theorem]{Lemma}
\newtheorem{Question}[theorem]{Question}
\newtheorem{remark}[theorem]{Remark}
\DeclareMathOperator{\tr}{Tr}

\newcommand{\R}{\mathbb R}

\newcommand{\E}{{\mathbb E}}

\newcommand{\diag}{\operatorname{Diag}}

\newcommand{\inr}[1]{\bigl< #1 \bigr>}
\newcommand{\eps}{\varepsilon}

\title{Covariance estimation with direction dependence accuracy}

\begin{document}
\author{Pedro Abdalla and Shahar Mendelson}
\maketitle

\begin{abstract}
We construct an estimator $\widehat{\Sigma}$ for covariance matrices of unknown, centred random vectors X, with the given data consisting of N independent measurements $X_1,...,X_N$ of X and the wanted confidence level. We show under minimal assumptions on X, the estimator performs with the optimal accuracy with respect to the operator norm. In addition, the estimator is also optimal with respect to direction dependence accuracy: $\langle \widehat{\Sigma}u,u\rangle$ is an optimal estimator for $\sigma^2(u)=\mathbb{E}\langle X,u\rangle^2$ when $\sigma^2(u)$ is ``large".
\end{abstract}

\section{Introduction}
The goal of this article is to study the covariance estimation problem with direction dependence accuracy. To give some idea of what one should expect, let us first describe the situation in the one dimensional case.

Let $x \in \mathbb{R}$ be a centred random variable with a finite fourth moment and consider $\kappa$ for which
\begin{equation} \label{ineq:moment_equivalence_onedim}
(\mathbb{E}x^4)^{1/4}\le \kappa (\mathbb{E}x^2)^{1/2}.
\end{equation}
Let $\phi:\mathbb{R}^N \times [0,1] \to \mathbb{R}$ be a procedure whose goal is to guess the variance of $x$: It receives as data $N$ independent copies of $x$ and a parameter $\delta \in (0,1)$, and returns a real number
$$
\widehat{\nu}= \phi(x_1,...,x_n,\delta).
$$
The procedure's error is $\eps$ if for any centred random variable $x$, with probability at least $1-\delta$ with respect to the $N$-product measure endowed by $x$,
$$
|\widehat{\nu}-\mathbb{E}x^2| \leq \eps.
$$
It is straightforward to identify the best performance one can hope for. Let $p=\log(2/\delta)/2N$, set $\alpha>0$, and define the law of $x$ by
\begin{equation*}
x=
\begin{cases}
    -\alpha, \quad \text{with probability $\frac{p}{2}$},\\
     \ \ 0, \quad \text{with probability $1-p$},\\
    \ \  \alpha, \quad \text{with probability $\frac{p}{2}$}.
\end{cases}
\end{equation*}

Clearly, with probability $(1-p)^N \ge 2\delta$, all the $N$ independent copies  $x_1,\ldots,x_N$ satisfy $x_i=0$, and for such samples no procedure can distinguish $x$ and $0$. Hence, the error in one of the two cases is at least $\mathbb{E}x^2/2$. In other words,
\begin{equation*}
\varepsilon\ge \frac{1}{2}\mathbb{E}x^2= \frac{1}{2}\alpha^2p = \frac{1}{2\sqrt{2}}\sqrt{\mathbb{E}x^4}\sqrt{\frac{\log(2/\delta)}{N}},
\end{equation*}
implying that the best performance one can hope for is that with probability at least $1-\delta$,
\begin{equation}
\label{ineq:best_bound_onedim}
    |\widehat{\nu} - \mathbb{E}x^2|\le C\sqrt{\mathbb{E}x^4}\sqrt{\frac{\log(1/\delta)}{N}};
\end{equation}
here $C$ is an absolute constant, and in particular it is strictly positive and does not depend on $x$ or the procedure $\phi$.

Once this benchmark is set, consider cases in which \eqref{ineq:moment_equivalence_onedim} holds. Then \eqref{ineq:best_bound_onedim} implies that
\begin{equation} \label{eq:sharp_one_dimensional}
|\widehat{\nu} - \mathbb{E}x^2|\le C\kappa^2 \mathbb{E}x^2 \sqrt{\frac{\log(1/\delta)}{N}}.
\end{equation}
As it happens, the \emph{trimmed mean} or the \emph{median of means} satisfy \eqref{eq:sharp_one_dimensional} (see \cite{lugosi2019mean}), showing that this performance is indeed optimal.

Note that the assumption that $x$ has a finite fourth moment is necessary. Indeed, if $x$ does not have a finite fourth moment then $x^2$ does not have a finite variance, resulting in a slower optimal error rate (see \cite[Theorem 3.1]{devroye16}). It also follows from \cite[Theorem 3.1]{devroye16} that the norm equivalence \eqref{ineq:moment_equivalence_onedim} is necessary if the goal is to obtain an error that scales like $\sigma^2/\sqrt{N}$.

Given the bounds in the one-dimensional case, a natural question is how to construct an optimal multi-dimensional covariance estimation procedure. However, even identifying the best possible error that one can hope for is challenging. We will explore this question in what follows, and then design an estimator that is optimal in a sense we shall clarify immediately.

\vskip0.4cm

\begin{definition}
A \emph{covariance estimation procedure} in $\mathbb{R}^d$ is a function $\phi:\mathbb{R}^{d \times N} \times [0,1] \to \mathbb{R}^{d \times d}$. Given $0<\delta<1$ and $X_1,...,X_N$, --- $N$ independent copies of a centred random vector $X$, the procedure returns a positive semi-definite matrix
$$
\hat{\Sigma}= \phi(X_1,...,X_N,\delta).
$$
The procedure performs with accuracy $\eps$ if for any centred random vector $X$, with probability at least $1-\delta$ with respect to the $N$-product measure endowed by $X$,
$$
\|\hat{\Sigma}- \Sigma_X \|_{2 \to 2} \leq \eps.
$$
Here and in what follows, $\Sigma_X$ is the covariance matrix of $X$ and $\| \cdot \|_{2 \to 2}$ is the ($\ell_2 \to \ell_2$) operator norm.
\end{definition}

The current state of the art upper estimate on the performance of a covariance estimation procedure is from \cite{abdalla2022covariance,oliveira2022improved}. To formulate that result we need a suitable version of \eqref{ineq:moment_equivalence_onedim}.
\begin{definition}
\label{def:L_4-L_2-assumption}
A centred random vector $X\in \mathbb{R}^d$ satisfies an $L_4-L_2$ norm equivalence with constant $\kappa$ if for every $u\in S^{d-1}$,
\begin{equation} \label{eq:norm-equiv-def}
    (\mathbb{E}\langle X,u\rangle^4)^{1/4} \le \kappa (\mathbb{E}\langle X,u\rangle^2)^{1/2}.
\end{equation}
\end{definition}

\begin{remark}
Since the $L_4(X)$ norm dominates the $L_2(X)$ norm, $\kappa$ in \eqref{eq:norm-equiv-def} must be at least $1$.
\end{remark}

In \cite{abdalla2022covariance} and \cite{oliveira2022improved} the authors construct a covariance estimator which satisfies that, for every centred random vector $X$, with probability at least $1-\delta$ with respect to the $N$-product measure endowed by $X$,
\begin{equation}
\label{ineq:sharp_covariance_operatornorm}
    \|\widehat{\Sigma}-\Sigma_X\|_{2 \to 2} \le C\left( \lambda_1\sqrt{\frac{r(\Sigma)+\log(1/\delta)}{N}}\right),
\end{equation}
where $C$ is an absolute constant, $\lambda_1\ge\ldots\ge \lambda_d\ge 0$ are the (ordered) eigenvalues of the covariance matrix $\Sigma$, and $r(\Sigma)$ is its effective rank; that is,
\begin{equation*}
    r(\Sigma):=\frac{\sum_{i=1}^d\lambda_i}{\lambda_1}.
\end{equation*}

The fluctuations in \eqref{ineq:sharp_covariance_operatornorm} depend on the behaviour of the ``worst-case" one-dimensional marginal of $X$, that is, on
\begin{equation*}
\lambda_1\sqrt{\frac{\log(1/\delta)}{N}};
\end{equation*}
therefore \eqref{ineq:sharp_covariance_operatornorm} does not give useful information on the variance of individual one-dimensional marginals. Ideally, one would like to have a covariance estimator for which $\|\widehat{\Sigma}-\Sigma_X\|_{2 \to 2}$ is small, and for ``many" directions $u \in S^{d-1}$,
$$
|\langle (\widehat{\Sigma}- \Sigma_X) u,u \rangle |
$$
scales like the optimal one-dimensional estimator. In other words, setting $\sigma^2(u):=\mathbb{E}\langle X,u\rangle^2$,
\begin{equation}
\label{ineq:optimal_1d}
    |\langle (\widehat{\Sigma}-
\Sigma_X) u,u \rangle |\le C\sigma^2(u) \sqrt{\frac{\log(1/\delta)}{N}}.
\end{equation}
\vskip0.4cm
Our aim here is to explore the best possible behaviour of a ``directional sensitive" covariance estimator:
\begin{tcolorbox} \label{box:Question}
\begin{Question} \label{qu:main}
What is the best possible function $S(u):S^{d-1} \rightarrow \mathbb{R}_{+}$ for which there is an estimator $\widehat{\Sigma}$ that satisfies (under minimal assumptions on $X$) that with probability at least $1-\delta$, for every $u\in S^{d-1}$,
\begin{equation}
\label{eq:key_question}
    \langle (\widehat{\Sigma}-\Sigma_X)u,u\rangle \leq  C\sigma^2(u) \sqrt{\frac{\log(1/\delta)}{N}} + S(u) ?
\end{equation}
\end{Question}
\end{tcolorbox}
\noindent An answer to Question \ref{qu:main} would imply that whenever $\sigma^2(u)\sqrt{\log(1/\delta)/N}$ dominates $S(u)$, the multi-dimensional estimator is also an optimal one-dimensional estimator of the variance of the marginal $\langle X,u\rangle$.

Throughout the article we assume without loss of generality that:
\begin{tcolorbox}
\begin{assumption} \label{basic_assumption}
The random vector $X$ has a symmetric, absolutely continuous density, and an invertible covariance matrix $\Sigma$; in particular, $X$ is centred. In addition, $X$ satisfies an $L_4$-$L_2$ norm equivalence with constant $\kappa$.
\end{assumption}
\end{tcolorbox}

The fact that Assumption \ref{basic_assumption} can be made without loss of generality is standard: the first part is evident by a symmetrisation argument, and if the need arises, following a tensorisation with a small gaussian vector. The second assumption can be made because without it, the smallest error one can hope for is much larger than $\sim \sigma^2(u)/\sqrt{N}$.

Our main result is a sharp answer to Question \ref{qu:main}:
\begin{tcolorbox}
\begin{theorem}[Main Result]
\label{thm:main_result}
There exist an absolute constant $\kappa_E$ and a constant $\kappa_{up}(\kappa)$ depending only on $\kappa$ for which the following hold. For every $\delta\in (0,1)$, there exists an estimator $\widehat{\Sigma}(X_1,\ldots,X_N,\delta)$ satisfying, with probability at least $1-\delta$ with respect to the $N$-product measure endowed by $X$, that for every $u\in S^{d-1}$
\begin{equation}
\begin{split}
\label{ineq:master_upper_bound}
&|\langle (\widehat{\Sigma}-\Sigma_X)u,u\rangle|\\
&\le \kappa_{up}\left(\sigma^2(u)\sqrt{\frac{\log(1/\delta)}{N}}+ \max\left\{\sigma(u),\sqrt{\lambda_{\kappa_E{\log(1/\delta)}}}\right\}\sqrt{\frac{\sum_{i\ge \frac{\kappa_E}{2}\log(1/\delta)}\lambda_i}{N}}\right).
\end{split}
\end{equation}
\end{theorem}
\end{tcolorbox}
\begin{remark}
The reason behind the notation $\kappa_E$ will become clear in what follows.
\end{remark}

For a positive semi-definite matrix $\Sigma$ with (ordered) eigenvalues $\lambda_1\ge \ldots\ge \lambda_d$, let $\mathcal{P}_{\preceq\Sigma}$ be the class of all centred vectors $d$-dimensional vectors $X$ whose covariance $\Sigma_X$ satisfies
\begin{equation*}
u^T\Sigma_X u\le u^T\Sigma u \quad \text{for every $u\in S^{d-1}$}.
\end{equation*}
Our second main result complements Theorem \ref{thm:main_result}. We show that whenever it is possible to give a non-trivial estimate to \eqref{eq:key_question}, Theorem \ref{thm:main_result} is indeed optimal (up to an absolute constant):

\begin{tcolorbox}
\begin{theorem}[Lower Bound]
\label{thm:second_main_result} 
Let $\Sigma$ be a positive semi-definite matrix and set $N\ge 18$. Suppose that there exist a constant $\gamma$ and an estimator $\widehat{\Sigma}(x_1,\ldots,x_N,\delta)$ which satisfies that for every $0<\delta<1$ and for every $X\in \mathcal{P}_{\preceq\Sigma}$,
\begin{equation*}
\mathbb{P}\left(\exists u\in S^{d-1}:|\langle (\widehat{\Sigma}-\Sigma_X)u,u\rangle|\ge \gamma \sigma^2(u)\sqrt{\frac{\log(1/\delta)}{N}}+S(u)\right)\le \delta.
\end{equation*}
Then there exist constants $c_1,\kappa_0$ depending only on $\gamma$ and $\kappa$ for which the following dichotomy holds, for every $u\in S^{d-1}$: either
\begin{equation*}
    S(u)\ge c_1\lambda_1\sqrt{\frac{\log(1/\delta)}{N}},
\end{equation*}
or
\begin{equation}
\label{ineq:master_lower_bound}
\begin{split}
& S(u) \\
&\ge \kappa_{low}\left(\sigma^2(u)\sqrt{\frac{\log(1/\delta)}{N}}+ \max\left\{\sigma(u),\sqrt{\lambda_{\kappa_0{\log(1/\delta)}}}\right\}\sqrt{\frac{\sum_{i\ge \kappa_0\log(1/\delta)}\lambda_i}{N}}\right).
\end{split}
\end{equation}
\end{theorem}
\end{tcolorbox}

\vskip0.4cm
Before we continue any further, let us fix some notation. In what follows, $B_2^d$ is the Euclidean unit ball in $\mathbb{R}^d$. For a subspace $E \subset \mathbb{R}^d$, let $B_E = B_2^d \cap E$ be the Euclidean unit ball in $E$, and set $P_E$ to be the orthogonal projection onto $E$. $I_d$ denotes the $d\times d$ identity matrix.

Let $g$ be the standard gaussian random vector in $\mathbb{R}^d$ and define the gaussian mean-width of a set $T\subset \mathbb{R}^d$ by
\begin{equation*}
    \ell^{\ast}(T):=\mathbb{E}\sup_{t\in T}\langle g,t\rangle.
\end{equation*}
A set $T \subset \mathbb{R}^d$ is symmetric if for every $t\in T$ we have that $-t\in T$. Let $$
\|\cdot\|_{T}:=\sup_{t\in T}\langle t,\cdot\rangle
$$
and note that $\|\cdot\|_{T}$ is a norm if $T$ is symmetric and spans $\mathbb{R}^d$. For a vector $X$ taking values in $\mathbb{R}^d$, $D$ and $\partial D$ are the unit ball and the unit sphere in $(\mathbb{R}^d,L_2(X))$, respectively.

Finally, absolute constants are denoted by $C$, $c$, etc. These are positive numbers that are independent of any of the parameters of the problem. Their values may change from line to line. If a constant depends on some parameter $\rho$, we specify that by writing  $c(\rho)$. We use $a \lesssim b$ if there is an absolute constant $c$ for which $a \leq cb$, and $a \sim b$ if $c_1 b \leq a \leq c_2b$ for absolute constants $c_1$ and $c_2$. We denote by $\kappa,\kappa_1,\kappa_2,...$ constants whose values remain unchanged throughout this article.

\vskip0.3cm
This article is organized as follows. The proof of Theorem \ref{thm:second_main_result} is presented in Section \ref{sec:Lower_Bound}. 
The proof of Theorem \ref{thm:main_result} requires some technical machinery, that is developed in Sections \ref{sec:Oracle} and \ref{sec:Estimation_Orthogonal}. In Section \ref{sec:Oracle}, we construct a data dependent function $\widehat{\psi}$ that identifies for points in $S^{d-1}$ whether the corresponding marginals have ``large or small" variance (in a sense to be clarified in what follows). In Section \ref{sec:Estimation_Orthogonal}, we show how such a function $\widehat{\psi}$ can be used to decompose $\mathbb{R}^d$ to subspaces $E$ and $E^{\perp}$ corresponding to ``large"  and ``small" directions. The proof of Theorem \ref{thm:main_result} is relatively simple for directions in either $E$ or $E^{\perp}$. The more subtle part is proving the claim for directions that are a ``mixture", which is done in Section \ref{sec:Chaining} using a chaining argument.

\vskip0.3cm

\section{Lower Bounds} \label{sec:Lower_Bound}
As noted previously, identifying the best possible performance of a directional-dependent covariance estimator is nontrivial, and this section is devoted to that goal. Specifically, exploring how big the function $S$ has to be if one is to have any hope of establishing \eqref{eq:key_question} under the minimal assumptions on $X$ from Assumption \ref{basic_assumption}.

Recall that the eigenvalues of $\Sigma \in \mathbb{R}^{d\times d}$ are $\lambda_1\ge\ldots\ge \lambda_d$, and its effective rank is $r(\Sigma) =  \sum_{i=1}^d\lambda_i/\lambda_1$. Let $\mathcal{P}_{\Sigma}$ be the class of centred $d$-dimensional random vectors whose covariance matrix is $\Sigma$.
\vskip 0.1cm
Our starting point is a minimax lower bound on the performance of covariance estimation procedure from \cite[Theorem 2]{lounici2014high}.
\begin{theorem}
\label{lem:minimax_lounici_covariance}
There exist absolute constants $\beta$ and $\kappa_1$ for which the following holds. For any positive semi-definite matrix $\Sigma \in \mathbb{R}^{d\times d}$,
\begin{equation}
\label{ineq:minimax_lounici_covariance}
    \inf_{\widehat{\Sigma}}\sup_{X \in \mathcal{P}_{\Sigma}} \mathbb{P}\left(\|\widehat{\Sigma}-\Sigma_X\|_{2 \to 2} \ge \kappa_1 \lambda_1\sqrt{\frac{r(\Sigma)}{N}}\right)\ge \beta,
\end{equation}
where the infimum is over all covariance estimators $\widehat{\Sigma}$.
\end{theorem}

Thanks to Theorem \ref{lem:minimax_lounici_covariance} we can establish the next fact, which will prove to be useful in what follows.

\begin{corollary} \label{cor:minimax-modified}
Let $\kappa_1$ and $\beta$ as in Theorem \ref{lem:minimax_lounici_covariance}. For every dimension $d$, covariance estimation procedure $\phi$, and invertible, positive semi-definite matrix $\Sigma \in \mathbb{R}^{d \times d}$, there is a centred random vector $Y \in \mathcal{P}_{\Sigma}$ for which
$$
\mathbb{P} \left( \left\| \Sigma^{-1/2} \phi(Y_1,...,Y_N,\delta)\Sigma^{-1/2} - I_k \right\|_{2 \to 2} \geq \kappa_1 \sqrt{\frac{d}{N}} \right) \geq \beta.
$$
\end{corollary}

\begin{proof}
Let $\phi$ be a covariance estimation procedure. For any invertible, positive semi-definite matrix $A$, set
$$
\phi_A(X_1,...,X_N,\delta)=A^{-1} \phi(AX_1,...,AX_N,\delta) A^{-1},
$$
which is also a covariance estimation procedure in $\mathbb{R}^d$. By Theorem \ref{lem:minimax_lounici_covariance} for $\mathcal{P}_{I_d}$ (i.e., the set of random vectors whose covariance is the identity), there exists a random vector $Z=Z_A \in \mathcal{P}_{I_d}$ satisfying that with probability at least $\beta$,
\begin{equation} \label{eq:minimax-modified}
\|\phi_A(Z_1,...,Z_N,\delta)-I_d \|_{2 \to 2} = \|A^{-1}\phi(AZ_1,...,AZ_N,\delta)A^{-1}-I_d \|_{2 \to 2} \geq \kappa_1 \sqrt{\frac{d}{N}}.
\end{equation}
Next, consider an invertible positive semi-definite matrix $\Sigma$, let $Y=\Sigma^{1/2}Z$ and note that the covariance of $Y$ is $\Sigma$. Using \eqref{eq:minimax-modified} for $A=\Sigma^{1/2}$ and $Z=Z_A$, we have that
$$
\|\Sigma^{-1/2}\phi(Y_1,...,Y_N,\delta)\Sigma^{-1/2}-I_d \|_{2 \to 2} \geq \kappa_1 \sqrt{\frac{d}{N}},
$$
as claimed.
\end{proof}

Next, note that an estimate like \eqref{eq:key_question} is nontrivial only when
$$
S<\lambda_1\sqrt{\frac{\log(1/\delta)}{N}}.
$$
We will show that `beating' $\lambda_1\sqrt{\log(1/\delta)/N}$ (when that is possible), comes at a price: an additional `global term'. The argument presented here is essentially the same as in \cite[Proposition 1]{lugosi2020multivariate}, adapted to the covariance case. To formulate it, let $\beta$ and $\kappa_1$ be the constants from Corollary \ref{cor:minimax-modified}, and recall that $\partial D$ is the unit sphere in $(\mathbb{R}^d,L_2(X))$.

\begin{theorem} \label{prop_lowerbound_strongterm}
Suppose that there is a constant $\gamma$ for which the following holds. If there exists an estimator $\widehat{\Sigma}(x_1,\ldots,x_N,\delta)$ that satisfies for every $X\in \mathcal{P}_{\Sigma}$ and $\delta \le \min\{\beta,e^{-1}\}$, that
\begin{equation}
\label{ineq:strongterm_assumption}
\mathbb{P}\left( \exists u\in S^{d-1}: \ u^T(\widehat{\Sigma}(X_1,\ldots,X_N,\delta)-\Sigma_X)u\ge \gamma \sigma^2(u)\sqrt{\frac{\log(1/\delta)}{N}} +S\right) \le \delta,
\end{equation}
then there are constants $c_1$ and $\kappa_0(\gamma,\kappa_1)$ for which either
\begin{equation}
\label{ineq:StrongTerm}
  S \geq c_1\lambda_1\sqrt{\log(1/\delta)/N} \quad \text{or} \quad S\ge \frac{\kappa_1}{2}\sqrt{\lambda_{{ \kappa_0\log(1/\delta)}}}\sqrt{\frac{\sum_{i\ge \kappa_0\log(1/\delta)}\lambda_i}{N}}.
\end{equation}

\end{theorem}

\begin{remark}
In what follows we will encounter several terms of the form $C \log(1/\delta)$. To ease notation, we shall assume without loss of generality that these are integers.
\end{remark}

\begin{proof}
Let $(e_i)_{i\le d}$ be the canonical basis of $\mathbb{R}^d$ and set $c_1>0$. Assume without loss of generality that $\Sigma$ is an invertible diagonal matrix with eigenvectors $e_1,\ldots,e_d$, and consider the case $S \leq c_1 \lambda_1 \sqrt{\log(1/\delta)/N}$.

Let $\ell$ be the unique integer for which
$$
S\in (c_1\lambda_{\ell+1}\sqrt{\log(1/\delta)/N}, c_1\lambda_{\ell}\sqrt{\log(1/\delta)/N}]
$$
and set $U_\ell:=\text{span}(e_1,\ldots,e_\ell)$. It is evident that, for every $u\in B_2^d\cap U_\ell$,
\begin{equation} \label{ineq:first_inequality_for_S_proof}
\begin{split}
& S\le c_1\sigma^2(u)\sqrt{\log(1/\delta)/N} \quad \text{and}
\\
& \sup_{u\in U_\ell\cap S^{d-1}}\frac{u^T \sqrt{N}(\widehat{\Sigma}-\Sigma) u}{\sigma^2(u)}\le (c_1 + \gamma)\sqrt{\log(1/\delta)}.
 \end{split}
\end{equation}

Let $\widehat{\Sigma}_{U_\ell}$ be the principal $\ell\times \ell$ submatrix of $\widehat{\Sigma}$, set  $\Sigma_{U_\ell}:=\diag(\lambda_1,\ldots,\lambda_\ell)$ and let $I_\ell$ be the identity matrix on $U_\ell$.
Clearly, for every $u \in U_\ell \cap S^{d-1}$, $\sigma(u)=\|\inr{X,u}\|_{L_2} = \|\Sigma_X^{1/2} u\|_2$, $\Sigma_X^{1/2}u \in U_\ell$ and

\begin{equation} \label{ineq:affine_map}
\begin{split}
&\sup_{u\in U_\ell\cap S^{d-1}} \frac{1}{\sigma^2(u)}u^T(\widehat{\Sigma}-\Sigma_X)u  = \sup_{u \in U_\ell\cap S^{d-1}} \left(\left\langle\widehat{\Sigma}\frac{u}{\sigma(u)},\frac{u}{\sigma(u)}\right\rangle-1\right)
\\
&=\sup_{u \in U_\ell\cap S^{d-1}} \left(\left\langle\widehat{\Sigma}_{U_\ell}\frac{u}{\|\Sigma_{U_\ell}^{1/2}u\|_2},\frac{u}{\|\Sigma_{U_\ell}^{1/2}u\|_2}\right\rangle-1\right) = (*).
\end{split}
\end{equation}

Moreover,
\begin{equation*}
\left\{\frac{u}{\|\Sigma_X^{1/2} u\|_2 } : u \in U_\ell \cap S^{d-1}\right\} = U_\ell \cap \partial D = \{ \Sigma_{U_\ell}^{-1/2} v : v \in   U_\ell \cap S^{d-1}\},
\end{equation*}
and therefore,
$$
(*)=\sup_{v \in U_\ell\cap S^{d-1}} \left(\langle\widehat{\Sigma}_{U_\ell}\Sigma_{U_\ell}^{-1/2}v ,\Sigma_{U_\ell}^{-1/2}v\rangle-1\right) \le \|\Sigma_{U_\ell}^{-1/2}\widehat{\Sigma}_{U_\ell}\Sigma_{U_\ell}^{-1/2} - I_\ell \|_{2 \to 2}.
$$
By Corollary \ref{cor:minimax-modified} there is $X \in \mathcal{P}_\Sigma$ for which, with probability at least $\beta$,
$$
\|\Sigma_{U_\ell}^{-1/2}\widehat{\Sigma}_{U_\ell}\Sigma_{U_\ell}^{-1/2} - I_\ell \|_{2 \to 2} \geq \kappa_1 \sqrt{\frac{\ell}{N}},
$$
and thus
$$
\sqrt{N} \sup_{u\in U_\ell\cap S^{d-1}} \frac{1}{\sigma^2(u)}u^T(\widehat{\Sigma}-\Sigma_X)u \geq \kappa_1 \sqrt{\ell}.
$$
Recalling that by \eqref{ineq:first_inequality_for_S_proof}, with probability at least $1-\delta$,
$$
\sqrt{N} \sup_{u\in U_\ell\cap S^{d-1}} \frac{1}{\sigma^2(u)}u^T(\widehat{\Sigma}-\Sigma_X)u \leq (c_1+\gamma) \sqrt{\log(1/\delta)},
$$
and that $\delta \leq \beta$, we have that
$$
\ell \leq \left(\frac{c_1+\gamma}{\kappa_1}\right)\log(1/\delta).
$$

Next, by the choice of $\ell$,
\begin{equation}
\label{ineq:second_inequality_for_S_proof}
    \sup_{u\in U_\ell^{\perp}\cap S^{d-1}} u^T(\widehat{\Sigma}-\Sigma_X) u \le 2S.
\end{equation}
One may follow the same argument used previously, this time for $U_{\ell}^{\perp}$ instead of $U_\ell$; $\widehat{\Sigma}_{U_{\ell}^{\perp}} $ (the $d-\ell\times d-\ell$ submatrix formed by the last $d-\ell$ rows and columns of $\widehat{\Sigma}$) instead of $\widehat{\Sigma}_\ell$; and $\Sigma_{U_\ell^{\perp}}:=\diag(\lambda_{\ell+1},\ldots,\lambda_d)$ instead of $\Sigma_{U_\ell}$. Since $\Sigma_X^{1/2}$ maps $U_\ell^{\perp}$ onto itself,
\begin{equation*}
\sup_{u\in U_\ell^{\perp}\cap S^{d-1}} u^T(\widehat{\Sigma}-\Sigma_X) u = \sup_{v\in S^{n-\ell-1}} v^T(\widehat{\Sigma}_{d-\ell}- \Sigma_{U_\ell^{\perp}})v.
\end{equation*}
Moreover, $\delta\le e^{-1}$, and therefore $\ell+1\le \kappa_0\log(1/\delta)$. Finally, $\|\Sigma_{U_\ell^{\perp}}\|_{2\rightarrow 2}=\lambda_{\ell+1}$ and the effective rank satisfies $r(\Sigma_{U_\ell^{\perp}})=\sum_{i\ge \ell+1}\lambda_i/\lambda_{\ell+1}$. Setting $\kappa_0 := (c_1+\gamma)/\kappa_1+1$, it follows from \eqref{ineq:second_inequality_for_S_proof} and Theorem \ref{lem:minimax_lounici_covariance} that
\begin{equation*}
S\ge \frac{\kappa_1}{2}\sqrt{\lambda_{\ell+1}}\sqrt{\frac{\sum_{i\ge \ell+1}\lambda_i}{N}} \ge \frac{\kappa_1}{2}\sqrt{\lambda_{\kappa_0\log(1/\delta)}}\sqrt{\frac{\sum_{i\ge\kappa_0\log(1/\delta)}\lambda_i}{N}}.
\end{equation*}
\end{proof}

The lower bound from Theorem \ref{prop_lowerbound_strongterm} is not sharp. As it happens, another term is necessary. To see that, consider $\sup_{u\in S^{d-1}}\langle (\widehat{\Sigma}-\Sigma)u,u\rangle = \|\widehat{\Sigma}-\Sigma\|_{2\rightarrow 2}$. If the lower bound on $S$ were sharp, there would have been a covariance estimator $\widehat{\Sigma}$, for which, with probability at least $1-\delta$,
\begin{equation}
\label{ineq:impossible_error}
\|\widehat{\Sigma}-\Sigma\|_{2\rightarrow 2}
\lesssim \lambda_1\sqrt{\frac{\log(1/\delta)}{N}} + \sqrt{\lambda_{\kappa_0\log(1/\delta)}} \sqrt{\frac{\sum_{i\ge \kappa_0\log(1/\delta)}\lambda_i}{N}}.
\end{equation}
However, it follows from \eqref{ineq:optimal_1d} and \eqref{ineq:minimax_lounici_covariance} that the best one can hope for (in the uniform sense over $\mathcal{P}_{\Sigma}$) is
\begin{equation}
\label{ineq:lower_bound_impossible_error_preliminary}
\begin{split}
\|\widehat{\Sigma}-\Sigma\|_{2 \rightarrow 2} &\sim \sup_{u\in S^{d-1}}\sigma^2(u)\sqrt{\frac{\log(1/\delta)}{N}} + \sqrt{\lambda_1}\sqrt{\frac{\sum_{i=1}^d\lambda_i}{N}}\\
&=\lambda_1\sqrt{\frac{\log(1/\delta)}{N}} + \sqrt{\lambda_1}\sqrt{\frac{\sum_{i=1}^d\lambda_i}{N}}\\
&\ge \lambda_1\sqrt{\frac{\log(1/\delta)}{N}} + \sqrt{\lambda_{1}}\sqrt{\frac{\sum_{i\ge\kappa_0\log(1/\delta)}\lambda_i}{N}}.
\end{split}
\end{equation}
Thus, for any $\delta<\beta$, any estimator $\widehat{\Sigma}$ satisfies for some random vector in $\mathcal{P}_{\Sigma}$ that, with probability at least $\delta$,
\begin{equation}
\label{ineq:lower_bound_impossible_error}
\|\widehat{\Sigma}-\Sigma\|_{2 \rightarrow 2} \gtrsim \lambda_1\sqrt{\frac{\log(1/\delta)}{N}} + \sqrt{\lambda_1}\sqrt{\frac{\sum_{i\ge\kappa_0\log(1/\delta)}\lambda_i}{N}}.
\end{equation}
On the other hand, observe that if
\begin{equation}
\label{ineq:impossible_error_part_2}
    \lambda_1 \ll \frac{\sum_{i\ge\kappa_0\log(1/\delta)}\lambda_i}{\kappa_0\log(1/\delta)} \quad \text{and} \quad \lambda_{\kappa_0\log(1/\delta)} \ll \lambda_1,
\end{equation}
then \eqref{ineq:impossible_error_part_2} and \eqref{ineq:impossible_error} imply that there is an estimator $\widehat{\Sigma}$ which satisfies that with probability at least $1-\delta$
\begin{equation*}
\|\widehat{\Sigma}-\Sigma\|_{2\rightarrow 2} \ll \lambda_1\sqrt{\frac{\log(1/\delta)}{N}}+\sqrt{\lambda_{1}}\sqrt{\frac{\sum_{i\ge k_0\log(1/\delta)}\lambda_i}{N}};
\end{equation*}
that clearly violates \eqref{ineq:lower_bound_impossible_error}. Thus, an additional term is needed, and to identify that term, we require some more notation.
\vskip 0.4cm

Let $G \sim \mathcal{N}(\mu,\Sigma)$ be a gaussian vector with mean $\mu$ and covariance matrix $\Sigma$. Also, recall that
\begin{equation*}
    \ell^{\ast}(T):=\mathbb{E}\sup_{t\in T}\langle g,t\rangle,
\end{equation*}
where, as always, $g$ is the standard gaussian vector in $\mathbb{R}^d$. Set $T$ to be a symmetric set that spans $\mathbb{R}^d$, let $\|\cdot\|_{T}:=\sup_{t\in T}\langle t,.\rangle$, and therefore $\|\cdot\|_{T}$ is a norm.

The analysis presented here relies on a sharp lower bound from \cite{lugosi2019near,depersin2022optimal}; the version we use is from \cite[Theorem 3]{depersin2022optimal}. 
\begin{theorem}
\label{lemma:lower_bound_gaussianwidth_general_norm}
There exists an absolute constant $\kappa_2$ for which the following holds. Consider an estimator $\widehat{\mu}(x_1,\ldots,x_N,\delta)$ for the means of the gaussian random vectors $G$ with covariance matrix $\Sigma$. If $\widehat{\mu}$ satisfies that for every $\delta<1/4$ and $\mu\in \mathbb{R}^d$,
\begin{equation*}
   \mathbb{P}(\|\widehat{\mu}(G_1,\ldots,G_N,\delta)-\mu\|_T\ge r^{\ast}) \le \delta,
\end{equation*}
then
\begin{equation}
\label{ineq:lower_bound_gaussianwidth_general_norm}
    r^{\ast}\ge \kappa_2\left(\sup_{t\in T}\sqrt{\mathbb{E}\langle G-\mu,t\rangle^2}\sqrt{\frac{\log(1/\delta)}{N}} + \frac{1}{\sqrt{N}}\ell^{\ast}(\Sigma^{1/2}T)\right).
\end{equation}
\end{theorem}
We actually need a slightly stronger version of Theorem \ref{lemma:lower_bound_gaussianwidth_general_norm}, by restricting the set of `eligible means' $\mu$ to a Euclidean ball $RB_2^d$ of radius $R:=3\sqrt{\tr(\Sigma)/N}$ (rather than to the entire $\mathbb{R}^d$, as is the case in Theorem \ref{lemma:lower_bound_gaussianwidth_general_norm}). The argument follows the path used in \cite{depersin2022optimal}, by considering a prior distribution on $\mu$, namely, that $\mu \sim \mathcal{N}(0,N^{-1}\Sigma)$.

\begin{corollary}
\label{cor:lower_bound_gaussianwidth_general_norm}
There exists an absolute constant $\kappa_2'$ for which the following holds. Consider an estimator $\widehat{\mu}(x_1,\ldots,x_N,\delta)$ for the means of gaussian random vectors $G$ with covariance matrix $\Sigma$. If the estimator $\widehat{\mu}$ satisfies for every $\delta <1/4$ and $\mu \in RB_2^d$ that
\begin{equation}
\label{cor:lower_bound_gaussianwidth_general_norm_assumption}
   \mathbb{P}(\|\widehat{\mu}(G_1,\ldots,G_N,\delta)-\mu\|_2 \ge r^{\ast})\le \delta,
\end{equation}
then
\begin{equation}
\label{ineq:lower_bound_gaussianwidth_general_norm_corollary}
   r^{\ast}\ge \kappa_2'\sqrt{\frac{\tr(\Sigma)}{N}}.
\end{equation}
\end{corollary}
\noindent We leave the straightforward proof to the Appendix.

Now with all the necessary ingredients in place, we can present a lower bound that complements Theorem \ref{prop_lowerbound_strongterm}. Recall that $\kappa_0$ is the constant from Theorem \ref{prop_lowerbound_strongterm}, fix a subspace $E^{\perp}$ of co-dimension $\kappa_0\log(1/\delta)$, and define $P_{E^{\perp}}:\mathbb{R}^d\rightarrow E^{\perp}$ to be the orthogonal projection onto $E^{\perp}$. Also, for $u\in E\cap S^{d-1}$ let $\mathcal{P}_{\Sigma,u}'$ be the set of all random vectors $X$ taking values in $\mathbb{R}^d$ and whose covariance matrix $\Sigma_X$ satisfies
\begin{equation}
\label{eq:common_property_Sigma'}
u^T\Sigma_Xu = u^T\Sigma u \quad \text{and} \quad \tr(P_{E^{\perp}}\Sigma_X P_{E^{\perp}}^T) \le \tr(\Sigma_{E^{\perp}}).
\end{equation}

\begin{theorem}
\label{prop_lowerbound_mixedterm}
Let $\kappa_2'$ be as in Corollary \ref{cor:lower_bound_gaussianwidth_general_norm} and $\gamma\kappa_2'/3$. Set $N\ge 18$. Let $\widehat{\Sigma}(x_1,\ldots,x_N,\delta)$ be an estimator which satisfies that for every $X\in \mathcal{P}_{\Sigma,u}'$ and $\delta\in (0,1/8)$,
\begin{equation}
\label{ineq:prop_lowerbound_mixedterm_assumption}
\mathbb{P}\left(\sup_{v\in E^{\perp}\cap S^{d-1}}|\langle \widehat{\Sigma}(X_1,\ldots,X_N,\delta)v,v \rangle - \langle \Sigma_X v,v\rangle | \ge \gamma \sigma(u)\sqrt{\frac{\tr(\Sigma_{E^{\perp}})}{N}}\right) \le \delta.
\end{equation}
Then there is $X\in \mathcal{P}_{\Sigma,u}'$ for which, with probability at least $\delta$,
\begin{equation}
\label{ineq:prop_lowerbound_mixedterm_conclusion}
\begin{split}
&\sup_{v\in (E^{\perp}\cap S^{d-1})\cup \{0\}}|\langle \widehat{\Sigma}(X_1,\ldots,X_N,\delta)(u+v),u+v \rangle - \langle \Sigma_X (u+v),u+v\rangle | \\
&\ge \gamma \sigma(u)\sqrt{\frac{\tr(\Sigma_{E^{\perp}})}{N}}.
\end{split}
\end{equation}
\end{theorem}
\begin{proof}
The first step of the proof is to construct a suitable collection of centred vectors $X_{\mu} \in \mathcal{P}_{\Sigma,u}'$. To that end,  set 
$$
R':=3\sigma(u)\sqrt{\frac{\tr(\Sigma
_{E^{\perp}})}{N}},
$$ 
and fix a gaussian random vector $\Tilde{G}$, taking values in $\mathbb{R}^{|E^{\perp}|}$ and whose covariance matrix $\widetilde{\Sigma} \in \mathbb{R}^{|E^{\perp}|\times |E^{\perp}|}$ satisfies that
\begin{equation}
\label{ineq:trace_tilde_Sigma}
    \tr(\widetilde{\Sigma}) = \tr(\Sigma_{E^{\perp}}) - \frac{R'^2}{\sigma^2(u)}.
\end{equation}     
Since      
\begin{equation*}     
\tr(\Sigma_{E^{\perp}}) - \frac{R'^2}{\sigma^2(u)}     
     = \tr(\Sigma_{E^{\perp}}) - 9\frac{\tr(\Sigma_{E^{\perp}})}{N} \ge \frac{1}{2}\tr(\Sigma_{E^{\perp}})>0,
\end{equation*}
there is a valid choice of $\Tilde{G}$. Fix $\mu \in R' B_2^{E^\perp}$ and define $X_{\mu}$ through its projections on $E$ and $E^{\perp}$ as follows.

Firstly, let $\varepsilon$ be a Rademacher random variable (i.e., symmetric, $\{-1,1\}$ valued) that is independent of $\Tilde{G}$. The law of the orthogonal projection of $X_{\mu}$ onto $E^{\perp}$ is
\begin{equation*}
   P_{E^{\perp}}X_{\mu}:= \Tilde{G} + \frac{1}{\sigma(u)}\varepsilon \mu.
\end{equation*}
Secondly, to define the law of the orthogonal projection of $X_{\mu}$ onto $E$, let $u_2,\ldots,u_{|E|}$ be unit vectors for which the set $\{u,u_2,\ldots,u_{|E|}\}$ is an orthonormal basis for $E$. Consider independent Rademacher random variables $\varepsilon_2,\ldots,\varepsilon_{|E|}$ and the variances $\sigma_2,\ldots,\sigma_{|E|}$ of $\varepsilon_2,\ldots,\varepsilon_{|E|}$, respectively. The law of $P_{E}X_{\mu}$ is given by 
\begin{equation*}
    P_{E}X_{\mu} := \varepsilon \sigma(u)u + \sum_{i=2}^{|E|}\varepsilon_i \sigma_i u_i,
\end{equation*}
and by the natural embedding of $P_{E}X_{\mu}$ and $P_{E^{\perp}}X_{\mu}$ in $\mathbb{R}^d$, the law of $X_{\mu}$ is defined by $X_{\mu} := P_{E}X_{\mu} + P_{E^{\perp}}X_{\mu} $. Let us verify that $X_{\mu} \in \mathcal{P}_{\Sigma,u}'$. Clearly,
\begin{equation*}
    \langle X_{\mu},u\rangle = \langle P_{E}X_{\mu},u\rangle = \sigma(u)\varepsilon, 
\end{equation*}
and then $\mathbb{E}\langle X,u\rangle^2 = \sigma^2(u)=u^T\Sigma_X u$. It is straightforward to verify that \eqref{ineq:trace_tilde_Sigma} implies that the covariance matrix $\Sigma_{\mu}:=\mathbb{E}(X_{\mu}\otimes X_{\mu})$ 
satisfies that $\tr(P_{E^{\perp}}\Sigma_{\mu}P_{E^{\perp}}^T) \le \tr(\Sigma_{E^{\perp}})$, therefore $X_{\mu} \in \mathcal{P}_{\Sigma,u}'$.

Next, observe that for every vector $X_\mu$ with covariance matrix $\Sigma_{{\mu}}$, we have that
\begin{equation}
\label{eq:expansion_mixedterm_lower_bound}
\langle (\widehat{\Sigma_{{\mu}}}-\Sigma_{{\mu}})(u+v), u+v\rangle = \langle (\widehat{\Sigma_{{\mu}}}-\Sigma_{{\mu}})u, u\rangle + \langle (\widehat{\Sigma_{{\mu}}}-\Sigma_{{\mu}})v, v\rangle + 2\langle (\widehat{\Sigma_{{\mu}}}-\Sigma_{{\mu}})u, v\rangle,
\end{equation}
and if \eqref{ineq:prop_lowerbound_mixedterm_conclusion} does not hold, then it follows that for every $X_\mu$, with probability at least $1-\delta$,
\begin{equation}
\label{prop_lowerbound_mixedterm_part2}
\sup_{v\in (E^{\perp}\cap S^{d-1})\cup\{0\}}\langle (\widehat{\Sigma_{{\mu}}}-\Sigma_{\mu})(u+v), u+v\rangle \le  \gamma \sigma(u)\sqrt{\frac{\tr(\Sigma_{E^{\perp}})}{N}}.
\end{equation}
Moreover, since $\delta <1/8$, \eqref{ineq:prop_lowerbound_mixedterm_assumption} and \eqref{prop_lowerbound_mixedterm_part2} hold together with probability  at least 
$$
1-(\delta+\delta) > 3/4.
$$ 
On that event, invoking \eqref{eq:expansion_mixedterm_lower_bound},
\begin{equation}
\label{ineq:prop_lowerbound_mixedterm_contradiction}
\begin{split}
&\sup_{v\in E^{\perp}\cap S^{d-1}}|\langle (\widehat{\Sigma_{{\mu}}}-\Sigma_{{\mu}}) u,v\rangle| \le \frac{1}{2}|\langle (\widehat{\Sigma_{{\mu}}}-\Sigma_{{\mu}}) u,u\rangle|
\\
& +\frac{1}{2}\sup_{v\in E^{\perp}\cap S^{d-1}}\left(|\langle (\widehat{\Sigma_{{\mu}}}-\Sigma_{{\mu}}) v,v\rangle| + |\langle (\widehat{\Sigma_{{\mu}}}-\Sigma_{{\mu}}) (u+v),u+v\rangle|\right)
\\
& \le \sup_{v\in (E^{\perp}\cap S^{d-1})\cup \{0\} }|\langle (\widehat{\Sigma_{{\mu}}}-\Sigma_{{\mu}}) (u+v),u+v\rangle| + \frac{1}{2}\sup_{v\in E^{\perp}\cap S^{d-1}} |\langle (\widehat{\Sigma_{{\mu}}}-\Sigma_{{\mu}}) v,v\rangle| \\
&\le \gamma\sigma(u)\sqrt{\frac{\tr(\Sigma_{E^{\perp}})}{N}} + \frac{\gamma}{2}\sigma(u)\sqrt{\frac{\tr(\Sigma_{E^{\perp}})}{N}}\\
& \le \frac{1}{2} \kappa_2'\sigma(u)\sqrt{\frac{\tr(\Sigma_{E^{\perp}})}{N}}:=r^{\ast},
\end{split}
\end{equation}
where the last step follows from the fact that $\gamma <\kappa^\prime_2/3$. 

Thus, to establish \eqref{ineq:prop_lowerbound_mixedterm_conclusion}, it suffices to show that \eqref{ineq:prop_lowerbound_mixedterm_contradiction} cannot hold. Indeed, notice that
\begin{equation}
\label{eq:mixed_gaussian_vector}
\langle X,u\rangle P_{E^{\perp}}X = \sigma(u)\Tilde{G} + \mu \sim \mathcal{N}(\mu,\sigma^2(u)\widetilde{\Sigma}),
\end{equation}
and that \eqref{ineq:prop_lowerbound_mixedterm_contradiction} implies that for every $\mu\in R'B_2^{|E^{\perp}|}$
\begin{equation*}
\begin{split}
&\frac{3}{4}<\mathbb{P}\left(\sup_{v\in E^{\perp}\cap S^{d-1}}|\langle (\widehat{\Sigma}_{\mu}-\Sigma_{\mu}) u,v\rangle|\le r^{\ast}\right)\\
&= \mathbb{P}\left(\|P_{E^{\perp}}\widehat{\Sigma}_{\mu} u- \mathbb{E}\langle X_{\mu},u\rangle (P_{E^{\perp}}X_{\mu})\|_{\ell_2(E^{\perp})}\le r^{\ast}\right).
\end{split}
\end{equation*}
It follows that $P_{E^{\perp}}\widehat{\Sigma}_{\mu} u$ is a mean estimation procedure with respect to Euclidean norm in $(\mathbb{R}^{|E^{\perp}|},\|\cdot\|_{\ell_2(E^{\perp})})$. More accurately, invoking \eqref{eq:mixed_gaussian_vector} the procedure estimates with accuracy at least $r^{\ast}$ and confidence at least $3/4$ the mean of every gaussian random vector of the form $\langle X_{\mu},u\rangle P_{E^{\perp}}X_{\mu}$ that takes value in $\mathbb{R}^{|E^{\perp}|}$. This includes all gaussian random vectors in $\R^{|E^\perp|}$ that have mean $\mu \in R'B_2^{|E^\perp|}$ and covariance $\sigma^2(u)\Tilde{\Sigma}$. Therefore by Corollary \ref{cor:lower_bound_gaussianwidth_general_norm} followed by \eqref{ineq:trace_tilde_Sigma}, 
\begin{equation*}
    r^{\ast} \ge \kappa_2' \sigma(u)\sqrt{\frac{\tr(\Tilde{\Sigma})}{N}} \ge \kappa_2' \sigma(u)\sqrt{\frac{\tr(\Sigma_{E^{\perp}})}{2N}} ,
\end{equation*}
which clearly contradicts \eqref{ineq:prop_lowerbound_mixedterm_contradiction}.
\end{proof}

The proof of Theorem \ref{thm:second_main_result} follows easily from Theorem \ref{prop_lowerbound_strongterm} and Theorem \ref{prop_lowerbound_mixedterm}.
\noindent

\section{The $L_2$ Isomorphic Distance Oracle} \label{sec:Oracle}
Here we take the first steps towards the proof of Theorem \ref{thm:main_result}.
In what follows, all the constants will either be absolute, or depend only on the $L_4-L_2$ norm equivalence constant $\kappa$. However, to make it easier to keep track them, we will, at times, mention whether the constants depend on each other as well (i.e., even though $c=c(\kappa)$ and $C=C(\kappa)$, we may write $c^\prime=c^\prime(c,C)$).

Our aim here is to construct an $L_2$ (isomorphic) distance oracle: a (random) function $\widehat{\psi}:\mathbb{R}^d\times (\mathbb{R}^d)^N$ that satisfies the following: With high probability with respect to the $N$-product measure endowed by X, for every $u,v\in S^{d-1}$, $\widehat{\psi}$ returns a ``good guess'' of the distance $\|u-v\|_{L_2} = \left(\E \inr{X,u-v}^2\right)^{1/2}$. Naturally, the best that one can hope for is that
$$
c^{-1}\|u-v\|_{L_2}\le \widehat{\psi}(u-v)\le c\|u-v\|_{L_2},
$$
for a constant $c$ that is close to $1$. Such a distance oracle is \emph{almost isometric}, whereas for an \emph{isomorphic} distance oracle, the constant $c$ can be a large absolute constant.

As it happens, a weaker notion suffices for our purposes, as described in Proposition \ref{prop:L_2_Oracle}. In fact, we only use the distance oracle in the next section to identify a high-dimensional section of the Euclidean sphere $E\cap S^{d-1}$, on which $\sigma^2(u)$ is large, while for $ u\in E^{\perp}\cap S^{d-1}$, $\sigma^2(u)$ is small --- and the meaning of `large' and `small' is clarified in what follows. Finding an accurate estimate on $\sigma^2(u)$ that holds for every $u\in E$ is possible only if the dimension of $E$ is at most $\kappa_En$, for some suitable (small) constant $\kappa_{E}$ and the choice of $E$ is explained in Section \ref{sec:Estimation_Orthogonal}. We set $\rho$ to be a free parameter that ranges from $0$ to $\kappa_0$. The choice of $\rho$ is specified in what follows.

Set $n =\log(1/\delta)$, recall that $D$ is the unit ball in $(\mathbb{R}^d,L_2(X))$ and that the random vector $X$ satisfies an $L_4-L_2$ norm equivalence with constant $\kappa$. Also, recall that $\kappa_0$ is the constant from Theorem \ref{prop_lowerbound_strongterm} --- and is an absolute constant.

\begin{remark} \label{remark_N_large}
Our goal is to produce a covariance estimator that is also an optimal one-dimensional procedure for directions in a subspace spanned by the eigenvectors corresponding to the largest $\sim n$ singular values of $\Sigma$ (of course, without having prior knowledge on the identity of that subspace). A far weaker constraint --- that the performance of the estimator is better than the trivial one on that subspace---, has immediate consequences. Indeed, let $u\in S^{d-1}$ that satisfies $\sigma^2(u) = \lambda_{\kappa_0n}$. By \eqref{ineq:master_lower_bound}, the best the performance of $\widehat{\Sigma}$ as a one-dimensional estimator in the direction $u$ that one can hope for is that with probability at least $1-e^{-n}$,
\begin{equation*}
\begin{split}
& \langle (\widehat{\Sigma}-\Sigma_X)u,u \rangle \\
&\le \kappa_{low}\left(\lambda_{\kappa_0n}\sqrt{\frac{n}{N}}+ \sqrt{\lambda_{\kappa_0n}}\sqrt{\frac{\sum_{i\ge \kappa_0n}\lambda_i}{N}}\right).
\end{split}
\end{equation*}
Thus, $\widehat{\Sigma}$ is outperformed in the direction $u$ by the trivial estimator $0$ unless
\begin{equation} \label{eq:outperformed-by-0}
\kappa_{low}\sqrt{\lambda_{ \kappa_0n}}\sqrt{\frac{\sum_{i\ge\kappa_0n}\lambda_i}{N}} < \lambda_{\kappa_0n} \quad \text{and} \quad N \ge  \kappa_{low}^2n.
\end{equation}
Hence, to have any hope of obtaining a procedure that is `nontrivial' in the entire subspace spanned by the $\kappa_0 n$ "largest directions", \eqref{eq:outperformed-by-0} must hold. To that end, we assume that $N$ satisfies
\begin{equation*}
 N \geq \widetilde{\kappa} \max\left\{\frac{\sum_{i\ge\kappa_0n}\lambda_i}{\lambda_{ \kappa_0n}},n\right\},
\end{equation*}
for some constant $\widetilde{\kappa}$ that is specified in what follows.
\end{remark}

Next, let us describe the features of $\widehat{\psi}$ that we require.

\begin{proposition} \label{prop:L_2_Oracle}
For any $0<\rho \le \kappa_0$, there exist constants $\theta(\kappa), \kappa_3(\rho,\kappa)$, and a function $\widehat{\psi}:\mathbb{R}^{d}\times (\mathbb{R}^d)^N\rightarrow \mathbb{R}_{+}$ for which the following holds: if  $\widetilde{\kappa}\ge\kappa_3\theta^{-3/2}$ then with probability at least $1-2e^{-\theta^2N/16}$,
\begin{enumerate}
    \item For $u\in S^{d-1} \cap \sqrt{\lambda_{\rho n}} D$, $\widehat{\psi}(u) \le \frac{16}{\theta^2}\lambda_{\rho n}$.
     \item For $u\in B_2^{d}$ that satisfies $\sigma^2(u)\ge \lambda_{\kappa_0n}$, we have $\widehat{\psi}(u) \ge \frac{\theta^{3/2}}{4\sqrt{2}(1-\theta)}\kappa^2\sigma^2(u)$.
\end{enumerate}
\end{proposition}
Note that $\rho\le \kappa_0$, therefore every $u\in S^{d-1}$ satisfies either (1) or (2),(where for $u$ with $\lambda_{\kappa_0n}\le \sigma^2(u)\le \lambda_{\rho n}$ both (1) and (2) hold).
\begin{remark}
Since $\rho$ will be an absolute constant, $\theta$ and $\kappa_3$ --- and therefore $\tilde{\kappa}$ as well --- depend only on the norm equivalence constant $\kappa$.
\end{remark}

\vskip0.4cm
For $u\in S^{d-1}$, let $J_{+}(u)$ be the set of the $\theta N$ largest indices of $(\langle X_i,u\rangle^2)_{i=1}^N$. Define $\widehat{\psi}$ by
\begin{equation*}
    \widehat{\psi}(u) = \frac{1}{N(1-\theta)}\sum_{i\in [N]/J_+(u)} \langle X_i,u\rangle^2,
\end{equation*}
which, at least intuitively, should be close to $\sigma^2(u)$ for every $u\in S^{d-1}$.

\vskip0.4cm

\medskip

\noindent \textbf{Proof of Proposition \ref{prop:L_2_Oracle} (1).} The heart of the argument is showing that
\begin{equation} \label{eq:prop-L-2-1-main}
    \mathbb{E}\sup_{u\in S^{d-1}\cap \sqrt{\lambda_{ \rho n}}D}\frac{1}{N}\sum_{i=1}^N \mathds{1}\left(\langle X_i,u\rangle^2 \ge \frac{16}{\theta^2} \lambda_{ \rho n} \right) \le \frac{\theta}{2}.
\end{equation}
Indeed, set
\begin{equation*}
\xi(x):=
\begin{cases}
   \  0,\quad  \text{if} \ x<2\sqrt{\lambda_{\rho n}}/\theta \\
   \ \frac{\theta x}{2\sqrt{\lambda_{\rho n}}} -1, \quad \text{if} \ 2\sqrt{\lambda_{ \rho n}}/\theta\le x\le 4\sqrt{\lambda_{\rho n}}/\theta\\
    \ 1, \quad \text{otherwise}.
\end{cases}
\end{equation*}
Clearly,
$$
\mathds{1}\left(x\ge\frac{4\sqrt{\lambda_{\rho n}}}{\theta}\right)\le \xi(x) \le \mathds{1}\left(x\ge\frac{2\sqrt{\lambda_{\rho n}}}{\theta}\right)
$$
and $\xi$ is a $\theta/2\sqrt{\lambda_{ \rho n}}$-Lipschitz function that passes through the origin. By the Gin\'e-Zinn symmetrization theorem \cite{gine1984some} and the contraction principle for Bernoulli processes (see e.g. \cite{ledoux1991probability}), we have that
\begin{equation*}
\begin{split}
&\mathbb{E}\sup_{u\in S^{d-1}\cap \sqrt{\lambda_{ \rho n}}D}\frac{1}{N}\sum_{i=1}^N \mathds{1}\left(\langle X_i,u\rangle^2 \ge \frac{16}{\theta^2}\lambda_{ \rho n}\right)\\
&\le \frac{1}{N}\mathbb{E}\sup_{u\in S^{d-1}\cap \sqrt{\lambda_{ \rho n}}D}\left|\sum_{i=1}^N \xi(|\langle X_i,u\rangle|)\right|
\\
&\le \frac{2}{N}\mathbb{E}\sup_{u\in S^{d-1}\cap \sqrt{\lambda_{ \rho n}}D}\left|\sum_{i=1}^N \varepsilon_i\xi(|\langle X_i,u\rangle|)\right| + \sup_{u\in S^{d-1}\cap \sqrt{\lambda_{ \rho n}}D} \mathbb{E}\xi(|\langle X,u\rangle|)
\\
&\le \frac{\theta}{\sqrt{\lambda_{\rho n}}} \frac{1}{N}\mathbb{E}\sup_{u\in S^{d-1}\cap \sqrt{\lambda_{ \rho n}}D}\left|\sum_{i=1}^N \varepsilon_i\langle X_i,u\rangle\right|+\sup_{u\in S^{d-1}\cap \sqrt{\lambda_{ \rho n}}D} \mathbb{E}\xi(|\langle X,u\rangle|)\\
&= (I) + (II)
\end{split}
\end{equation*}
By Markov's inequality,
$$
(II) \leq \frac{\theta^2}{4\lambda_{\rho n}} \sup_{u\in S^{d-1}\cap \sqrt{\lambda_{ \rho n}}D}\sigma^2(u) < \frac{\theta}{4}.
$$
The estimate on (I) follows from the same arguments as in \cite{lugosi2020multivariate}: observe that $D=\Sigma^{-1/2} B_2^d$ and that the random vector $W=\Sigma^{-1/2}X$ is isotropic. Clearly, $\Sigma^{1/2}B_2^d \cap \sqrt{\lambda_{ \rho n}}B_2^d$ is contained in an ellipsoid $\mathcal{E}$ whose principal axes are proportional to $\min\{\sqrt{\lambda_i},\sqrt{\lambda_{ \rho n}}\}$, and setting $W_i:=\Sigma^{-1/2}X_i$, it is evident that there is an absolute constant $C_1$ for which
\begin{equation} \label{ineq_Rademacher_Complexity_ellipsoid}
\begin{split}
&\mathbb{E}\sup_{u\in B_2^d \cap \sqrt{\lambda_{ \rho n}}D} \frac{1}{N}\left|\sum_{i=1}^N \varepsilon_i\langle X_i,u\rangle \right|= \mathbb{E}\sup_{u\in B_2^d \cap \sqrt{\lambda_{ \rho n}}\Sigma^{-1/2}B_2^d} \frac{1}{N}\left|\sum_{i=1}^N \varepsilon_i\langle W_i,\Sigma^{1/2}u\rangle \right|
\\
&=\mathbb{E}\sup_{v \in \Sigma^{1/2}B_2^d \cap \sqrt{\lambda_{ \rho n}}B_2^d} \frac{1}{N}\left|\sum_{i=1}^N \varepsilon_i\langle W_i,v\rangle \right|
\\
& \le C_1\mathbb{E}\sup_{v\in \mathcal{E}}\frac{1}{N}\left|\sum_{i=1}^N \varepsilon_i\langle W_i,v\rangle \right|
\\
&\leq  C_1\sqrt{\frac{\sum_{i=1}^d \min\{\lambda_i,\lambda_{ \rho n}\}}{N}}
\\
&\leq C_1\left(\sqrt{\frac{\rho n\lambda_{\rho n}}{N}} + \sqrt{\frac{\sum_{i\ge\rho n} \lambda_{i}}{N}}\right)\\
&\leq C_1\left(\sqrt{\frac{\rho n\lambda_{\rho n}}{N}} + \sqrt{\frac{\lambda_{\rho n}(\kappa_0-\rho)n+\sum_{i\ge\kappa_0 n} \lambda_{i}}{N}}\right)
\end{split}
\end{equation}
By the assumption on $N$ (see Remark \ref{remark_N_large}) and since $\lambda_{\rho n}\ge\lambda_{\kappa_0 n}$, we may set
\begin{equation}
\label{first_choice_kN}
\widetilde{\kappa}\ge 128C_1^2\max\{\kappa_0,1\},
\end{equation}
thus ensuring that
\begin{equation*}
\frac{\theta}{\sqrt{\lambda_{ \rho n}}}\mathbb{E}\sup_{u\in S^{d-1}\cap \sqrt{\lambda_{ \rho n}}D} \frac{1}{N}\left|\sum_{i=1}^N \varepsilon_i\langle X_i,v\rangle \right| \leq \frac{\theta}{8} + \frac{\theta}{8} = \frac{\theta}{4},
\end{equation*}
and establishing \eqref{eq:prop-L-2-1-main}.

The wanted high probability estimate follows from the bounded difference inequality (see e.g. \cite{vershynin2018high}): for every $\theta \in (0,1)$, with probability at least $1-e^{-\theta^2N/8}$,
\begin{equation*}
 \sup_{u\in S^{d-1}\cap \sqrt{\lambda_{ \rho n}}D}\sum_{i=1}^N \mathds{1}\left(\langle X_i,u\rangle^2\ge \frac{16}{\theta^2}\lambda_{ \rho n}\right) \le \theta N.
\end{equation*}
Thus, on that event
\begin{equation*}
    \sup_{u\in S^{d-1}\cap \sqrt{\lambda_{ \rho n}}D} \widehat{\psi}(u) \leq \frac{16}{\theta^2}\lambda_{ \rho n}.
\end{equation*}
\vskip0.3cm

Turning to the second part of the proof, our starting point here is from \cite{koltchinskii2015bounding} (see also \cite[Theorem 5.3 ]{mendelson2015learning}).
\begin{theorem}
\label{thm:smallball}
Let $A>0$ be a positive number. For $u\in B_2^{d}$, set
\begin{equation*}
    Q_{A}(u):= \mathbb{P}(|\langle X,u\rangle|\ge A).
\end{equation*}
If $T\subset B_2^d$ then with probability at least $1-e^{-t^2/2}$,
\begin{equation}
\label{ineq:smallballmethod}
\inf_{u\in T }\frac{A}{N}\sum_{i=1}^N \mathds{1}(|\langle X_i,u\rangle|\ge A) \ge A\inf_{u\in T}Q_{2A}(u) - \frac{2}{N}\mathbb{E}\sup_{u\in T}\left|\sum_{i=1}^N \varepsilon_i\langle X_i,u\rangle\right| - A \frac{t}{\sqrt{N}}.
\end{equation}
\end{theorem}
\medskip
\noindent \textbf{Proof of Proposition \ref{prop:L_2_Oracle} (2).}
For a  vector $x\in \mathbb{R}^N$, set  $x^{\ast}$ to be the non-decreasing rearrangement of $(|x_i|)_{i=1}^N$, in particular, $x_1^{\ast}\le \ldots\le x_N^{\ast}$. Also, recall that $\partial D$ denotes the unit sphere in $(\mathbb{R}^d,L_2(X))$.

\vskip0.3cm
\noindent \textbf{Step 1 - nontrivial $Q_A(u)$:}
We will apply Theorem \ref{thm:smallball} with $T=B_2^d\cap \sqrt{\lambda_{ \kappa_0n }} \partial D$. The first step is to ensure that $\inf_{u\in T}Q_{A}(u)$ is strictly positive. More accurately, it follows from the Paley-Zygmund inequality and the $L_4-L_2$ norm equivalence that
\begin{equation*}
    \mathbb{P}(\langle X,u\rangle^2\ge (1-\alpha)\sigma^2(u))\ge (1-\alpha)^2 \frac{(\mathbb{E}\langle X,u\rangle^2)^2}{\mathbb{E}\langle X,u\rangle^4} \ge \kappa^{-4} (1-\alpha)^2 \ge 2\theta,
\end{equation*}
provided that $0<\alpha < 1-\sqrt{2\theta}\kappa^2$. One may consider $\theta \in (0,\kappa^{-4}/2)$, and thus $\alpha$ is strictly positive. Setting $A:=\kappa^2\sqrt{\theta/8}$, we have that
\begin{equation}
\label{ineq:step1_smallball}
\inf_{u\in B_2^d \cap \sqrt{\lambda_{\kappa_0n}} \partial D}\mathbb{P}(\langle X,u\rangle^2\ge 4A \lambda_{ \kappa_0n}) \ge 2\theta.
\end{equation}
\textbf{Step 2 - reduction to an empirical process:}
Recall that
$$
\widehat{\psi}(u)= \frac{1}{(1-\theta)N}\sum_{i\in [N]/J_+(u)} \langle X_i,u\rangle^2 = \frac{1}{(1-\theta)N}\sum_{i=1}^{(1-\theta)N} (\langle X_i,u\rangle^{\ast})^2,
$$
and the goal here is to establish a ``user-friendly" lower bound on
\begin{equation*}
\inf_{u \in B_2^d\cap \sqrt{\lambda_{\kappa_0n }}D}\frac{1}{(1-\theta)N}\sum_{i=1}^{(1-\theta)N} (\langle X_i,u\rangle^{\ast})^2.
\end{equation*}

Observe that
\begin{equation*}
\begin{split}
\sum_{i=1}^{(1-\theta)N}(\langle X_i,u\rangle^{\ast})^2 &\ge A\lambda_{\kappa_0n} \sum_{i=1}^{(1-\theta)N}\mathds{1}\left(\langle X_i,u\rangle^{\ast}\ge \sqrt{A\lambda_{ \kappa_0n}}\right)\\
&\ge A\lambda_{ \kappa_0n}\left(\sum_{i=1}^{N}\mathds{1}(\langle X_i,u\rangle\ge \sqrt{A\lambda_{ \kappa_0n}})-\theta N\right).
\end{split}
\end{equation*}
Therefore,
\begin{equation}
\label{ineq:smallball_before_technical_step}
\begin{split}
&\frac{1}{(1-\theta)N}\sum_{i=1}^{(1-\theta)N} (\langle X_i,u\rangle^{\ast})^2\\
&\ge \frac{A\lambda_{ \kappa_0n}}{(1-\theta)N}\sum_{i=1}^{N}\mathds{1}\left(\langle |\langle X_i,u\rangle|\ge \sqrt{A\lambda_{ \kappa_0n}}\right) - \frac{\theta A\lambda_{ \kappa_0n}}{(1-\theta)},
\end{split}
\end{equation}
and in particular,
\begin{equation*}
\widehat{\psi}(u) \geq \frac{A\lambda_{ \kappa_0n}}{(1-\theta)N}\sum_{i=1}^{N}\mathds{1}\left(\langle |\langle X_i,u\rangle|\ge \sqrt{A\lambda_{ \kappa_0n}}\right) - \frac{\theta A\lambda_{ \kappa_0n}}{(1-\theta)}.
\end{equation*}

Hence, to show that for some well-chosen constant $c$, with probability $1-e^{-c\theta N}$,
\begin{equation}
\label{ineq:smallball_finalgoal}
\inf_{u\in B_2^{d}\cap \sqrt{L\lambda_{ \kappa_0n}}\partial D}\widehat{\psi}(u) \ge \frac{\theta}{2(1-\theta)}A \lambda_{ \kappa_0n},
\end{equation}
it suffices to find a high probability event on which
\begin{equation}
\label{ineq:smallball_finaltechnicalstep}
\frac{1}{N}\sum_{i=1}^N\mathds{1}\left(\langle |\langle X_i,u\rangle|\ge \sqrt{A\lambda_{ \kappa_0n}}\right) \ge \left(\theta+\frac{\theta}{2}\right).
\end{equation}
\medskip

\textbf{Step 3 - application of Theorem \ref{thm:smallball}:}
Let $\mathcal{E}$ be an ellipsoid whose principal axes are proportional to $\min\{\sqrt{\lambda_i},\sqrt{\lambda_{\kappa_0n}}\}$, and in particular it contains $B_2^d\cap \sqrt{\lambda_{\kappa_0n}}D $. Setting $A':=\sqrt{A\lambda_{ \kappa_0n}}$, it follows from  Theorem \ref{thm:smallball} that with probability at least $1-e^{-t^2/2}$,
\begin{equation*}
\begin{split}
&\inf_{u\in B_2^d\cap \sqrt{\lambda_{\kappa_0n}} \partial D }\frac{A'}{N}\sum_{i=1}^N \mathds{1}\left(|\langle X_i,u\rangle|\ge A' \right) \\
&\ge A'\inf_{u\in B_2^d\cap \sqrt{\lambda_{ \kappa_0n}}\partial D}\mathbb{P}(|\langle X_i,u|\rangle\ge 2A') - \frac{2}{N}\mathbb{E}\sup_{u\in \mathcal{E}}\sum_{i=1}^{N}\varepsilon_i\langle X_i,u\rangle - A'\frac{t}{\sqrt{N}} \\
& := (a) +(b) + (c).
\end{split}
\end{equation*}
The estimate on (b) follows from the same arguments used in \eqref{ineq_Rademacher_Complexity_ellipsoid}: there is an absolute constant $C_1$ for which
\begin{equation*}
\frac{2}{N}\mathbb{E}\sup_{u\in \mathcal{E}}\sum_{i=1}^{N}\varepsilon_i\langle X_i,u\rangle \leq \sqrt{\lambda_{\kappa_0n}}\frac{2C_1}{\sqrt{\widetilde{\kappa}}}(\sqrt{\kappa_0}+1),
\end{equation*}
and then we may choose
\begin{equation}
\label{second_choice_kN}
\widetilde{\kappa} \ge \frac{64C_1^2(\sqrt{\kappa_0}+1)^2}{A\theta^2} = \frac{128\sqrt{2}C_1^2(\sqrt{\kappa_0}+1)^2}{\kappa^2\theta^{3/2}},
\end{equation}
implying that
\begin{equation*}
\frac{2}{N}\mathbb{E}\sup_{u\in \mathcal{E}}\sum_{i=1}^{N}\varepsilon_i\langle X_i,u\rangle \leq \frac{\theta A'}{4}.
\end{equation*}
Next, invoking \eqref{ineq:step1_smallball} and setting $t=\theta\sqrt{N}/4$, we have that with probability at least $1-e^{-\theta^2N/16}$
\begin{equation*}
(a)+(b)+(c)\ge A'\left(2\theta- \frac{\theta}{4} - \frac{\theta}{4}\right) = A'\frac{3}{2}\theta =\frac{3}{2}\theta \sqrt{A\lambda_{\kappa_0n}}.
\end{equation*}

Finally, by homogeneity, if $\sigma'>\lambda_{\kappa_0n}$ then on that event,
\begin{equation*}
\inf_{u\in B_2^{d}:\sigma^2(u) = \sigma'} \widehat{\psi}(u) =\inf_{u\in B_2^d:\sigma^2(u) =\lambda_{ \kappa_0n}} \widehat{\psi}(u) \frac{\sigma'}{\lambda_{\kappa_0n}} \ge  \frac{\theta}{2(1-\theta)}A \sigma'.
\end{equation*}

\section{Estimation in Orthogonal Subspaces} \label{sec:Estimation_Orthogonal}
The next step towards an answer to Question \ref{qu:main} is constructing a subspace $E\subset \mathbb{R}^d$ of dimension proportional to $n$, and obtaining an accurate estimate on $\sigma^2(u)=\mathbb{E}\langle X,u\rangle^2$ for every $u\in E\cup E^{\perp}$.

As always, the random vector $X$ satisfies Assumption \ref{basic_assumption}, and in particular the $L_4-L_2$ norm equivalence with constant $\kappa$:
$$
(\mathbb{E}\langle X,u\rangle^4)^{1/4}\le \kappa(\mathbb{E}\langle X,u\rangle^2)^{1/2} \ \ {\rm  for \ every \ } u\in S^{d-1}.
$$
Recall the constant $\kappa_0$ from Theorem \ref{prop_lowerbound_strongterm} and the constants $\rho,\theta,\kappa_3$ from Proposition \ref{prop:L_2_Oracle} (describing properties of the distance oracle), and that $\widetilde{\kappa}\ge \kappa_3\theta^{-3/2}$; thus the assumptions of Proposition \ref{prop:L_2_Oracle} are satisfied. As noted previously, all these constants depend only on $\kappa$.

\begin{lemma}
\label{lem:construction_E_Eperp}
There exist a constant $\kappa_{4}(\theta)$ for which the following holds. For any $0<\rho\le \kappa_0$, there exists a subspace $E=E(X_1,\ldots,X_N)$ of dimension $r=\rho n$ for which, with probability at least $1-2e^{-\theta^2N/16}$ with respect $N$-product measure endowed by $X$,
\begin{equation}
\label{ineq:construction_E_Eperp}
\sup_{u\in E^{\perp}\cap S^{d-1}} \sigma^2(u) \le \kappa_4 \lambda_{r}.
\end{equation}
\end{lemma}
The proof is based on the properties of the distance oracle that was constructed in the previous section.
\begin{proof}
Let $\Omega$ be the event from Proposition \ref{prop:L_2_Oracle} and set $\widehat{\psi}$ to be the distance oracle constructed there. Let $\mathcal{Z}$ be a subspace of codimension $r$ that minimizes
\begin{equation}
\label{ineq:construction_E_Eperp_optimization}
    \sup_{u \in \mathcal{Z} \cap S^{d-1}} \widehat{\psi}(u),
\end{equation}
and set $E:=\mathcal{Z}^{\perp}$.

Note that the subspace spanned by the $d-r$ eigenvectors corresponding to the smallest eigenvalues of $\Sigma$ is a feasible solution of \eqref{ineq:construction_E_Eperp_optimization}; hence, by the first part of Proposition \ref{prop:L_2_Oracle} and the fact that $\rho\le \kappa_0$, we have that
\begin{equation*}
\sup_{u\in \mathcal{Z} \cap S^{d-1}}\widehat{\psi}(u) \le \frac{16}{\theta^2(\kappa)}\lambda_{\rho n}.
\end{equation*}
Invoking the second part of Proposition \ref{prop:L_2_Oracle}, if $u\in S^{d-1}$ satisfies that $\sigma^2(u)\ge \lambda_{\kappa_0n}$, then  $$\widehat{\psi}(u)\ge \frac{\theta^{3/2}}{4\sqrt{2}(1-\theta)}\sigma^2(u).$$ Thus, there is a constant $\kappa_4(\theta)$ for which, if $u\in S^{d-1}$ satisfies that $\sigma^2(u)\ge \kappa_4 \lambda_{\rho n}$, then
\begin{equation*}
    \widehat{\psi}(u) \ge \kappa_4 \frac{\theta^{3/2}}{4\sqrt{2}(1-\theta)} \lambda_{\rho n} > \sup_{u\in \mathcal{Z}\cap S^{d-1}}\widehat{\psi}(u).
\end{equation*}
Therefore, on the event $\Omega$ if $u\in \mathcal{Z}\cap S^{d-1}$ then $\sigma^2(u)\le \kappa_4 \lambda_{\rho n}$.
\end{proof}

In what follows, we consider a data splitting approach: The first part of the given sample will be used for the distance oracle and the construction of the subspaces $E$ and $E^{\perp}$, and the second part for covariance estimation. Once the subspace $E$ is chosen, it is relatively straightforward to estimate $\sigma^2(u)$ for $u\in E$ or $u\in E^{\perp}$.
\begin{remark}
As we explain in what follows, dealing with an arbitrary $u$ rather than $u\in E\cup E^{\perp}$ requires an additional argument.
\end{remark}
As always, set $n=\log(1/\delta)$ and recall that the subspace $E$ from Lemma \ref{lem:construction_E_Eperp} has dimension $r=\rho n$. From here on, set $\rho=\kappa_E$ to be an absolute constant $0<\rho\le \kappa_0$ to be named in next proposition.

\begin{proposition} \label{prop:estimation_E_Eperp}
There exist constants $\kappa_5(\kappa,\kappa_0)$, $\kappa_6(\kappa)$ and $\kappa_E$ for which the following holds. For every $n\ge 1$, there is a function $\widehat{\nu}(u):S^{d-1}\times (\mathbb{R}^d)^{N}\rightarrow \mathbb{R}_{+}$ with the following properties. If $r=\kappa_{E}n$ then with probability at least $1-2e^{-n}$ with respect to the $N$-product measure endowed by $X$,
\begin{enumerate}
    \item For every $u\in E^{\perp}\cap S^{d-1}$,
    \begin{equation*}
        |\widehat{\nu}(u) - \sigma^2(u)| \le \kappa_5\left(\sqrt{\lambda_{r}}\sqrt{\frac{\sum_{i\ge r/2}\lambda_i}{N}}\right).
    \end{equation*}
    \item  For every $u\in E\cap S^{d-1}$,
     \begin{equation*}
        |\widehat{\nu}(u) - \sigma^2(u)| \le \kappa_6\sigma^2(u)\sqrt{\frac{n}{N}}.
    \end{equation*}
\end{enumerate}
\end{proposition}
\vskip0.3cm
\subsection{Proof of Proposition \ref{prop:estimation_E_Eperp} $(1)$: Estimation in $E^{\perp}$}
Because of the nature of the bound in (1), the problem of estimating $\mathbb{E}\langle X,u\rangle^2$ in $E^{\perp}$ can be recast as a covariance estimation problem; specifically, estimating the covariance of the random vector $P_{E^{\perp}}X$ with respect to the operator norm. To that end, recall that $X$ satisfies the $L_4-L_2$ norm equivalence \eqref{eq:norm-equiv-def} with constant $\kappa$, and our starting point is the following result, which is based on \cite[Theorem 1]{abdalla2022covariance}.
\begin{theorem} \label{abdalla_nikita_estimator}
There exists an constant $C(\kappa)$ for which the following holds. For every $n \ge 1$, there is an estimator $\widehat{\Sigma}(X_1,\ldots,X_N,n)$ satisfying that with probability at least $1-e^{-n}$,
\begin{equation}
\label{ineq:abdalla_nikita_estimator}
\|\widehat{\Sigma}-\Sigma_X\|_{2 \to 2} \le C(\kappa)\|\Sigma\|_{2 \to 2}\sqrt{\frac{n+r(\Sigma)}{N}}.
\end{equation}
\end{theorem}
Clearly, if $X_1,\ldots,X_N$ are independent copies of $X$ then $P_{E^{\perp}}X_1,\ldots,P_{E^{\perp}}X_N$ are independent copies of $P_{E^{\perp}}X$. Denote the covariance of $P_{E^\perp} X$ by $\Sigma_{E^\perp}$ and observe that for every $u\in E^{\perp}$
\begin{equation*}
    \mathbb{E}\langle X,u\rangle^2  = \mathbb{E}\langle X,P_{E^{\perp}}u\rangle^2 = \langle \Sigma_{E^{\perp}} u,u\rangle.
\end{equation*}
Hence, one may apply Theorem \ref{abdalla_nikita_estimator} to the random vector $P_{E^{\perp}}X$.

By Lemma \ref{lem:construction_E_Eperp},
\begin{equation}
\label{ineq:operatornorm_Eperp}
    \|\Sigma_{E^{\perp}}\|_{2 \to 2} =\sup_{u\in S^{d-1}}\mathbb{E}\langle P_{E^{\perp}}X,u\rangle^2 = \sup_{u\in S^{d-1}\cap E^{\perp}} \mathbb{E}\langle X,u\rangle^2 \leq \kappa_4 \lambda_{r}.
\end{equation}
Moreover, to estimate the effective rank $r(\Sigma_{E^{\perp}})$ it suffices to control $\tr(\Sigma_{E^{\perp}})$. To that end, let $v_1,\ldots,v_d$ be the orthonormal basis of the eigenvectors corresponding to the eigenvalues $\Sigma$. It is straightforward to verify that
$$
\frac{r}{2}\lambda_{r}\le \sum_{i=r/2}^{r} \lambda_i \le \sum_{i\ge r/2}\lambda_i,
$$
and that
$$
\sum_{i=1}^{r} \mathbb{E}\langle P_{E^{\perp}}X,v_i\rangle^2 \leq r \|\Sigma_{E^\perp}\|_{2 \to 2}.
$$
Therefore,
\begin{align} \label{ineq:trace_Eperp}
\tr(\Sigma_{E^{\perp}}) = & \sum_{i=1}^d \mathbb{E}\langle P_{E^{\perp}}X,v_i\rangle^2 \le r\|\Sigma_{E^{\perp}}\|_{2\rightarrow 2} +  \sum_{i>r}\mathbb{E}\langle X,P_{E^{\perp}}v_i\rangle^2  \\
&\le r\kappa_4\lambda_{r} + \sum_{i>r}\mathbb{E}\langle X,P_{E^{\perp}}v_i\rangle^2  \nonumber
\\
& \le 2\kappa_4\sum_{i\ge r/2}\lambda_i.
\end{align}
Thus, 
\begin{equation*}
\|\Sigma_{E^{\perp}}\|_{2\rightarrow 2}\sqrt{\frac{r(\Sigma_{E^{\perp}})}{N}} =  \sqrt{\|\Sigma_{E^{\perp}}\|_{2\rightarrow 2}}\sqrt{\frac{\tr(\Sigma_{E^{\perp}})}{N}} \le \kappa_4\sqrt{2\lambda_r \cdot\frac{\sum_{i\ge r/2}\lambda_i}{N}},
\end{equation*}
and it follows from Theorem \ref{abdalla_nikita_estimator} that the estimator $\widehat{\Sigma}$ construct there satisfies that with probability at least $1-e^{-n}$,
\begin{equation*}
\begin{split}
(*):=&\sup_{u\in S^{d-1}}|\langle \widehat{\Sigma}u,u\rangle - \langle \Sigma_{E^{\perp}} u,u\rangle|\\
&\leq C(\kappa) \kappa_4\left(\lambda_{r}\sqrt{\frac{n}{N}} + \sqrt{2\lambda_{r} \cdot \frac{\sum_{i\ge r/2}\lambda_i}{N}}\right).
\end{split}
\end{equation*}
Recalling that $r=\kappa_E n$,
\begin{equation*}
\begin{split}
\lambda_{r}\sqrt{\frac{n}{N}} &= \sqrt{\frac{\lambda_{r}}{N}} \sqrt{n\lambda_{r}} \le \sqrt{\frac{\lambda_{r}}{N}} \sqrt{\frac{2}{\kappa_E} \sum_{i=r/2}^{r}\lambda_i } \\
&\leq \sqrt{\frac{2}{\kappa_E}} \sqrt{\lambda_{r}}\sqrt{\frac{\sum_{i\ge r/2}\lambda_i}{N}},
\end{split}
\end{equation*}
and setting $\kappa_5 = C(\kappa)\kappa_4\sqrt{2}\max\{\kappa_E^{-1/2},1\}$, we obtain that
\begin{equation*}
(*) \leq \kappa_5\sqrt{\lambda_{r} \frac{\sum_{i\ge r/2}\lambda_i}{N}}.
\end{equation*}
The first part of Proposition \ref{prop:estimation_E_Eperp} follows by setting $\widehat{\nu}(u) = \langle\widehat{\Sigma}u,u\rangle$.

\subsection{Proof of Proposition \ref{prop:estimation_E_Eperp} (2): Estimation in $E$}
Next, consider $N$ independent copies of the vector $P_{E}X$. The problem is to estimate $\mathbb{E}\langle X,u\rangle^2$ uniformly in $E\cap S^{d-1} \equiv S_{E}$, which is a Euclidean sphere in a subspace of dimension $r$. The construction we present resembles the one from \cite{depersin2020robust} and is based on a median-of-means procedure.

Set $m=N/n$, and without loss of generality assume that $m$ is an integer. For every $u\in  S^{d-1}$,
\begin{enumerate}
\item Split $\langle P_{E}X_1,u\rangle,\ldots, \langle P_{E}X_N,u\rangle$ into disjoint blocks of indices $B_1,\ldots,B_n$, each one of cardinality $m$.
\item Set
\begin{equation*}
    \widehat{\sigma}_{B_i}(u) = \frac{1}{n}\sum_{i\in B_i}\langle P_{E}X_i,u\rangle^2,
\end{equation*}
i.e, $\widehat{\sigma}_{B_i}(u)$ is the empirical mean of $\langle P_{E}X,u\rangle^2$ on each block.
\item Let $\widehat{\nu}(u)$ be a median of the numbers $\widehat{\sigma}_{B_1}(u),\ldots,\widehat{\sigma}_{B_m}(u)$.
\end{enumerate}
The estimates on the performance of this median-of-means procedure follows the ideas from \cite[Theorem 3.2]{mendelson2021approximating}, implying that the number of blocks in which $\widehat{\sigma}_{B_i}$ behaves badly is small. To quantify this statement, let us reformulate \cite[Theorem 3.2]{mendelson2021approximating} in a slightly more general scenario. The statement and its proof are based on VC theory\footnote{For more information on VC dimension, we refer the reader to \cite{vershynin2018high,ledoux1991probability}.}. In what follows, for a function $f$ and a sample $X_1,\ldots,X_N$, the empirical mean of $f$ is denoted by $$\mathbb{P}_N(f):=\frac{1}{N}\sum_{i=1}^N f(X_i).$$
\begin{theorem} \label{thm:mendelson_VC}
There exist absolute constants $c_1$ and $c_2$ for which the following holds. Let $\mathcal{F}$ be a class of $\{0,1\}$-valued functions whose VC-dimension is at most $s$ and set
\begin{equation*}
   c_1\frac{s}{n}\log\left(\frac{n}{s}\right) \le \Delta \le \frac{1}{2}.
\end{equation*}
Then with probability at least $1-e^{-c_2\Delta n}$, for every $f \in \mathcal{F}$,
$$\mathbb{P}_n(f) \le \frac{3}{2}\mathbb{P}(f) + 2\Delta.$$
\end{theorem}
\begin{remark}
In the case that interests us, it turns out that $s \sim r$ and $\Delta$ is a well-chosen absolute constant.
\end{remark}

\vskip0.4cm
\noindent To apply Theorem \ref{thm:mendelson_VC}, consider the class of functions
\begin{equation*}
     \mathcal{F}:=\left\{\mathds{1}\left(\left|\sum_{i=1}^m \langle X_i,u\rangle^2 -\sigma^2(u)\right| \ge t\right), \ u\in B_{E}, \ t\in \mathbb{R}\right\}.
\end{equation*}
Thus, each function in $\mathcal{F}$ depends on $m$ independent copies of $P_EX$- a random vector in a space of dimension $r$.

One may show that
\begin{equation} \label{ineq:bound_VC_F}
VC(\mathcal{F})\le 4({\rm dim}(E)+1) \log_2(8e) \leq 19r.
\end{equation}
Indeed, \eqref{ineq:bound_VC_F} is based on Warren's Lemma \cite{warren1968lower}, and the version used here is from \cite[Corollary 2]{depersin2020robust}:
\begin{lemma} \label{lem:Warrenslemma}
Let $\mathbb{R}_{\ell}[x_1,\ldots,x_k]$ be the ring of polynomials over $\mathbb{R}$ of degree at most $\ell$ and in the variables $x_1,\ldots,x_k$. The class of indicators
\begin{equation*}
    \mathcal{H}:=\left\{\mathds{1}(P(x_1,\ldots,x_k))\ge 0) \ : \  P(x_1,\ldots,x_k)\in \mathbb{R}_{l}[x_1,\ldots,x_k] \right\},
\end{equation*}
satisfies that $$VC(\mathcal{H}) \leq 2k\log_2(4e\ell).$$
\end{lemma}
\noindent Now, set
\begin{equation*}
\mathcal{H}_{+}:=\left\{\mathds{1}\left(\sum_{i=1}^m \langle X_i,u\rangle^2 -\sigma^2(u) \ge t \right) \ : \  u\in B_{E}, \ t \in \mathbb{R}\right\}
\end{equation*}
and note that
$$
P(u_1,\ldots, u_{r},t):=\sum_{i=1}^m \langle X_i,u\rangle^2 -\sigma^2(u)-t
$$
is a polynomial in the variables $(u_1,\ldots,u_r,t)$ and has degree at most $2$. Therefore
$$
VC(\mathcal{H}_{+})\le 2(r+1)\log_2(8e).
$$

Using an identical argument for
\begin{equation*}
\mathcal{H}_{-}:=\left\{\mathds{1}\left(\sum_{i=1}^m \langle X_i,u\rangle^2 -\sigma^2(u) \le -t \right) \ : \   u\in B_{E}, \ t \in \mathbb{R}\right\},
\end{equation*}
we have that
$$
VC(\mathcal{H}_{-})\le 2(r+1)\log_2(8e)
.$$
The estimate \eqref{ineq:bound_VC_F} follows immediately using that $VC(\mathcal{F})\le VC(\mathcal{H}_{+})+VC(\mathcal{H}_{-})$.

Next, it is evident from Markov's inequality and the $L_4-L_2$ norm equivalence that for any $u \in B_E$,
\begin{equation*}
\mathbb{P}\left(\left|\frac{1}{m}\sum_{i=1}^m\langle X_i,u\rangle^2 -\sigma^2(u)\right| \ge \sqrt{8}\frac{\kappa^2\sigma^2(u)}{\sqrt{m}}\right) \le \frac{1}{8}.
\end{equation*}

\medskip

\noindent  Setting
\begin{equation*}
\mathcal{H}_1:=\left\{\mathds{1}\left(\left|\sum_{i=1}^m \langle X_i,u\rangle^2 -\sigma^2(u)\right| \ge \sqrt{8}\frac{\kappa^2\sigma^2(u)}{\sqrt{m}}\right) \ : \  u\in B_{E}\right\},
\end{equation*}
invoking Theorem \ref{thm:mendelson_VC} and recalling \eqref{ineq:bound_VC_F}, it follows that with probability at least $1-e^{-c_2\Delta n}$, for every $h\in \mathcal{H}_1$,
\begin{equation*}
    \mathbb{P}_n(h) \le \frac{3}{2}\mathbb{P}(h) + 2\Delta \le \frac{3}{16}+2\Delta,
\end{equation*}
provided that
$$
(*):=c_1\frac{19r\log(n/19r)}{n} \le \Delta \le \frac{1}{2}.
$$
Recall that $r = \kappa_{E}n$, and we may choose $\kappa_{E}(c_1)$ sufficiently small to ensure that $(*) \le 1/32$. Hence, setting $\Delta = 1/32$, it is evident that with probability at least $1-e^{-c_2n/32}$, for every $h\in \mathcal{H}_1$,
\begin{equation*}
    \mathbb{P}_n(h) \le \frac{1}{4}.
\end{equation*}
As $\widehat{\nu}(u)$ is a median of $\left(\frac{1}{m} \sum_{i \in B_j} \langle X_i,u\rangle^2 \right)_{i=1}^n$, then on the same event, for every $u\in E\cap S^{d-1}$,
\begin{equation*}
   \left|\widehat{\nu}(u) -\sigma^2(u) \right| \leq \sqrt{8}\kappa^2 \frac{\sigma^2(u)}{\sqrt{m}} = \sqrt{8}\kappa^2 \sigma^2(u)\sqrt{\frac{n}{N}}.
\end{equation*}
Finally re-scaling $n\rightarrow c_2n/32$ and setting $\kappa_{6}:=16\kappa^2c_2^{-1/2}$, we have that with probability $1-e^{-n}$,
\begin{equation*}
   \left|\widehat{\nu}(u) -\sigma^2(u) \right| \leq \kappa_6\sigma^2(u)\sqrt{\frac{n}{N}}.
\end{equation*}
\section{Estimation in the Euclidean Sphere}
\label{sec:Chaining}
After addressing the problem of estimating $\mathbb{E}\langle X,u\rangle^2$ when $u\in E$ or $u\in E^{\perp}$, for the subspace $E$ that was constructed using the first half of the given sample,
let us turn to the general case, in which $u$ is an arbitrary point in $S^{d-1}$. Consider the decomposition $u=P_Eu+P_{E^{\perp}}u$ and the expansion
\begin{equation*}
\mathbb{E}\langle X,u\rangle^2 = \mathbb{E}\langle X,P_Eu\rangle^2 + \mathbb{E}\langle X,P_{E^{\perp}}\rangle^2 + 2\mathbb{E}\langle X,P_Eu\rangle\langle X,P_{E^{\perp}}u\rangle.
\end{equation*}
Clearly, both $\mathbb{E}\langle X,P_Eu\rangle^2$ and $\mathbb{E}\langle X,P_Eu\rangle^2$ can be estimated accurately, thanks to Proposition \ref{prop:estimation_E_Eperp}. All that remains is to answer the following question: How one may estimate the correlation term $\mathbb{E}\langle X,P_Eu\rangle\langle X,P_{E^{\perp}}u\rangle$ with a direction-dependent accuracy? This question is answered in Theorem \ref{prop:Estimation_correlation_term}. To that end, recall Lemma \ref{lem:construction_E_Eperp} and Proposition \ref{prop:estimation_E_Eperp}, and assume that the events from both claims hold.
\begin{theorem} \label{prop:Estimation_correlation_term}
There are constants $c$ and $\kappa_{7}(\kappa,\kappa_E)$ for which the following holds. For every $n\ge 1$ there is a function $\widehat{\nu}(u):S^{d-1}\times (\mathbb{R}^d)^N\rightarrow \mathbb{R}_{+}$ satisfying that with probability at least $1-e^{-cn}$ with respect to the $N$-product measure endowed by $X$, for every $u\in S^{d-1}$
\begin{equation*}
|\widehat{\nu}(u) - \mathbb{E}\langle X,P_{E}u\rangle\langle X,P_{E^{\perp}}u\rangle| \le \kappa_{7}\left(\max\left\{\sigma(u),\sqrt{\lambda_{ r }}\right\}\sqrt{\frac{\sum_{i\ge r/2}\lambda_i}{N}}\right).
\end{equation*}
\end{theorem}
The argument is based on \emph{generic chaining} and Talagrand's majorizing measure theorem. 

\subsection{Construction of an Admissible Sequence}
Recall that $(E,d_{L_2})$ is a metric space, and that the metric $d_{L_2}$ is endowed by the norm $\|t\|_{L_2}:=\mathbb{E}\langle X,t\rangle^2$. Recall that $B_E = B_2^d\cap E$ and denote the $L_2$ diameter of a set $T\subset B_{E}$ by $\Delta_{L_2}(T):=\Delta(T)$.

In what follows, we consider a sequence of partitions $(T_{s})_{s\ge 0}$ of $B_{E}$ that is increasing: every $T \in T_{s+1}$ is contained in some $T^\prime \in T_{s}$.
\begin{definition}
An admissible sequence of $(B_{E},d_{L_2})$ is an increasing sequence of partitions of $B_{E}$, $\{T_{s}:s\ge 0\}$, satisfying that for every integer $s\ge 1$, $|T_{s}|\le 2^{2^{s}}$ and $|T_0|=1$. For an integer $s_0\ge 0$, set
\begin{equation*}
\gamma_{2,s_0}(B_{E},d_{L_2}) = \inf \sup_{t\in B_{E}}\sum_{s\ge s_0}2^{s/2}\Delta(T_s(t)),
\end{equation*}
where the infimum is taken with respect to all admissible sequences of $B_{E}$ and $T_{s}(t) \in T_{s}$ is the unique subset that contains $t$.
\end{definition}
A key component in the proof of Theorem \ref{prop:Estimation_correlation_term} is the Talagrand's majorizing measure theorem (see \cite[Theorem B3.3]{talagrand2014upper}). To formulate the result, let $G$ be the centered gaussian vector whose covariance is $\Sigma$, set $(G_{t})_{t\in B_E}:= \langle G,t\rangle$ and put $S:=\mathbb{E}\sup_{t\in B_{E}}G_{t}$.
\begin{theorem} \label{thm_TalagrandMM_Appendix}
Set $L>0$. For each $k\ge 1$ consider $\delta_k>0$ for which
\begin{equation*}
    \forall t \in B_{E}, \quad \mathbb{E}\sup_{s\in B_{E}:d_{L_2}(t,s)\le \delta_k}|G_t-G_s|\le 2^{-k}S,
\end{equation*}
and let $s_k$ be an integer that satisfies
\begin{equation*}
    LS2^{-s_k/2 -k} \le \delta_k.
\end{equation*}
Then there is a constant $L_1(L)$ and an admissible sequence $(T_{s})_{s \geq s_k}$ of $B_{E}$ for which
\begin{equation} \label{eq_TalagrandMM_Appendix-0}
    \forall k\ge 1, \quad  \sup_{t\in B_{E}}\sum_{s\ge s_k}2^{s/2}\Delta(T_{s}(t)) \le L_1 S 2^{-k}.
\end{equation}
\end{theorem}

\begin{remark}
It should be stressed that Theorem \ref{thm_TalagrandMM_Appendix} is constructive: if one has access to the distances $d_{L_2}(s,t)$ for every $s,t \in B_E$, there is a procedure whose output is the admissible sequence that satisfies \eqref{eq_TalagrandMM_Appendix-0}. However, in our case the metric $d_{L_2}$ is not known. What saves the day is Proposition \ref{prop:estimation_E_Eperp}: one can estimate $\mathbb{E}\langle X,u\rangle^2$ uniformly over all $u\in B_{E}$ and construct a metric $\widehat{d}_{L_2}$ for which
$$
c^{-1}(\kappa)\widehat{d}_{L_2}(t,s)\le d_{L_2}(t,s)\le c(\kappa)\widehat{d}_{L_2}(t,s)
$$
for some constant $c(\kappa)$. Hence it is possible to construct an admissible sequence $\widehat{T}_s$ of $(B_{E},\widehat{d}_{L_2})$ for which
\begin{equation} \label{eq:good-admissible-1}
\sup_{t\in B_{E}}\sum_{s\ge s_k}2^{s/2}\Delta(\widehat{T}_{s}(t)) \le C(\kappa) 2^{-k}S.
\end{equation}
\end{remark}

Now that we have an admissible sequence at our disposal, let $\pi_{s}t$ be the nearest point to $t$ in $\widehat{T}_{s}$ with respect to $\widehat{d}_{L_2}$, and set $\Delta_{s}t = \pi_{s}t-\pi_{s-1}t$. The sequence of partitions $\widehat{T}_{s}$ is increasing, and thus for every $t\in B_{E}$, both $\pi_{s}t$ and $\pi_{s-1}t$ belong to $T_{s-1}$. By re-indexing the sum, inequality \eqref{eq:good-admissible-1} and the fact that $\|\Delta_s t\|_{L_2} \sim \|\Delta_s t\|_{\widehat{L}_2}$ imply that
\begin{equation}
\label{ineq:generic_chaining_bound_with_widehatT_s_and_L_2}
\sup_{t\in B_{E}}\sum_{s\ge s_k}2^{s/2}\|\Delta_st\|_{L_2} \leq C_1(\kappa)2^{-k}S
\end{equation}
for a well-chosen constant $C_1(\kappa)$.

Next, one may decompose $P_{E}u$ along the chain $\widehat{T}_{s}$ by writing
\begin{equation*}
    P_E u = \sum_{s=s_0+1}^{s_1} \Delta_s P_Eu + \pi_{s_0}P_Eu + (P_E u -\pi_{s_1}P_Eu)
\end{equation*}
for integers $s_0,s_1$ that are defined by
\begin{equation}
\label{def:choice_s0_s1}
s_0 = \widetilde{\kappa}_0\log_2n \ \ \ {\rm  and} \ \ \  s_1= \lceil\log_2N \rceil,
\end{equation}
and $\widetilde{\kappa}_0$ that is a constant to be specified in what follows.

Note that
\begin{equation}
\begin{split}
&\mathbb{E}\langle X,P_{E}u\rangle\langle X,P_{E^{\perp}}u\rangle = \left(\sum_{s=s_0+1}^{s_1} \mathbb{E}\langle \Delta_s P_Eu,X\rangle \langle X,P_{E^{\perp}}u\rangle \right) \\
& + \mathbb{E}\langle X,\pi_{s_0}P_Eu\rangle\langle X,P_{E^{\perp}}u\rangle
+ \mathbb{E}\langle X,P_E u -\pi_{s_1}P_Eu\rangle\langle X,P_{E^{\perp}}u\rangle\\
&:= (I)+(II)+(III),
\end{split}
\end{equation}
and the proof of Theorem \ref{prop:Estimation_correlation_term} follows by controlling these three terms.

\begin{remark}
It is important to keep in mind that there is access to each term in the decomposition of $P_Eu$ along the chain: the admissible sequence $(\widehat{T}_s)$ is constructed from the given data, as is the subspace $E$, $\Delta_s P_E u$, $\pi_{s_0} P_E u$ and $\pi_{s_1}P_E u$. Hence, one may invoke mean estimation procedures and obtain sharp estimates on $(I)$, $(II)$ and $(III)$.
\end{remark}

We also require the following lemma. Recall that $r=\kappa_E n$ is the dimension of the subspace $E$.
\begin{lemma} \label{lem:application_talagrand_MM}
There exists a constant $\kappa_{8}(\kappa_E,\kappa)$ for which
\begin{equation}
\label{ineq:application_talagrand_MM}
\sup_{t\in B_{E}}\sum_{s\ge s_0+1} 2^{s/2}\|\Delta_{s} t\|_{L_2}\le \kappa_{8}\sqrt{\lambda_{r}} \sqrt{n}.
\end{equation}
\end{lemma}
\noindent We postpone the proof of Lemma \ref{lem:application_talagrand_MM} to the end of this section.

\subsection{Estimating the Term II}
To simplify notation, set $\pi_{s_0}P_Eu = z$. Thus, $z \in \widehat{T}_{s_0}$, ensuring that there are at most $2^{2^{s_0}}$ choices of $z$. Then, a bound on (II) follows by fixing a vector $z$ and estimating (with the necessary high probability) $\mathbb{E}\langle X,z\rangle\langle X,P_{E^{\perp}}u\rangle$ uniformly over all $u\in E^{\perp}\cap S^{d-1}$. Invoking the union bound over all possible choices of $z$ yields the wanted bound.

Let $\mathcal{N}(\mathcal{F},\|\cdot\|_{L_2},\varepsilon)$ be the $L_2$-covering numbers of $\mathcal{F}$, namely the minimal number of open balls of radius $\varepsilon>0$ with respect to the $\|\cdot\|_{L_2}$ needed to cover the set $\mathcal{F}$. Our starting point is \cite[Theorem 2]{lugosi2019near}:

\begin{theorem} \label{thm:uniform_MOM}
Consider a class of real-valued functions $\mathcal{F}$ on $\mathbb{R}^d$, and for $m >0$ let
 \begin{equation*}
    p_m(\eta):=\sup_{f\in \mathcal{F}}\mathbb{P}\left(\left|\frac{1}{m}\sum_{i=1}^mf(X_i)-\mathbb{E}f\right|\ge \eta\right).
\end{equation*}
Let $c_1,c_2$ and $c_3$ be well-chosen absolute constants and set $\eta_0,\eta_1,\eta_2>c_1\eta_1/\sqrt{m}$ that satisfy the following:
\begin{enumerate}
    \item $p_m(\eta_0) \le 0.05$.
    \item $\log \mathcal{N}(\mathcal{F},\|.\|_{L_2},\eta_1) \le c_2 n\log(e/p_m(\eta_0))$.
    \item $\mathbb{E}\sup_{w\in \overline{W}}|\sum_{i=1}^N\varepsilon_iw(X_i)|\le c_3\eta_2N$,
\end{enumerate}
where $W:= \mathcal{F}-\mathcal{F}$ and  $\overline{W}:=\{w-\mathbb{E}w:w\in W\}$.

Then, there are an absolute constant $c$ and an estimator $\widehat{F}:\mathcal{F} \times (\mathbb{R}^d)^N \to \mathbb{R}$ that satisfies
\begin{equation*}
\mathbb{P}\left(\forall f\in \mathcal{F}, \  | \widehat{F}(f) - \mathbb{E}f(X)| \le \eta_0 + \eta_2\right)\ge 1-e^{-cn}.
\end{equation*}
\end{theorem}
\noindent Consider a fixed $z \in B_E$, and we will apply Theorem \ref{thm:uniform_MOM} to the class of functions $\mathcal{F}_z:=\{ \langle X,z\rangle \langle X,\theta\rangle: \theta\in B_{E^{\perp}} \}$. Set $m=N/n$ and note that by Chebyshev's inequality,
\begin{equation*}
p_m(\eta_0)\le \sup_{\theta \in B_{E^\perp}} \frac{1}{\eta_0^2 m} \|\langle X,z\rangle\langle X,\theta\rangle\|_{L_2}^2.
\end{equation*}
Thus, if we choose
$$
\eta_0 = C_{0}\sup_{\theta \in B_{E^{\perp}}}\frac{1}{\sqrt{m}}\|\langle X,z\rangle\langle X,\theta\rangle\|_{L_2},
$$
for a suitable choice of an absolute constant $C_{0}$, it follows that $p_m(\eta_0)<0.05$.

Next, to control $(2)$ one has to estimate the $L_2$-covering numbers of the class $\mathcal{F}_z$. Thanks to the $L_4-L_2$ norm equivalence, it is enough to control the cardinality of a maximal $\kappa^4\eta_1/\sigma(z)$-separated subset of $B_{E^{\perp}}$ with respect to the $L_2$ metric. To that end, recall that $G$ is a centered gaussian vector with covariance matrix $\Sigma$. By Sudakov's minoration inequality \cite{talagrand2014upper},

\begin{equation*}
\sqrt{\log\mathcal{N}\left(B_{E^{\perp}},\|\cdot\|_{L_2},\frac{\kappa^4\eta_1}{\sigma(z)}\right)}\frac{\kappa^4\eta_1}{\sigma(z)} \leq C \mathbb{E}\sup_{u\in S^{d-1}\cap \sqrt{\lambda_{r}}D} \langle G,t\rangle \leq C_1\sqrt{\sum_{i\ge r}\lambda_i}
\end{equation*}
for an absolute constant $C_1$.

Thus, we may choose $$\eta_1 := \frac{C_1}{c_2}\kappa^{-4}\sigma(z)\sqrt{\frac{\sum_{i\ge r}\lambda_i}{n}}$$ to ensure that $(2)$.

Finally, we proceed to verify $(3)$ by showing that there exists an absolute constant $C_{2}$ for which
\begin{equation}
\label{ineq:Rademacher_complexity_term_eta_3_final}
\begin{split}
&\mathbb{E}\sup_{\mathcal{F}-\mathcal{F}}\left|\sum_{i=1}^N \varepsilon_i (\langle X_i,z\rangle\langle X_i,\theta_{1}-\theta_{2}\rangle - \mathbb{E}\langle X_i,z\rangle\langle X_i,\theta_{1}-\theta_{2}\rangle )\right|\\
&\le C_{2}\kappa^2\sqrt{N}\sigma(z)\left(\sqrt{\lambda_{r}}+ \sqrt{\sum_{i\ge r/2}\lambda_i}\right)\le c_3\eta_2 N,
\end{split}
\end{equation}
by choosing 
$$
\eta_2 \ge \frac{C_2}{c_3}\kappa^2\sigma(z){\frac{\sqrt{\lambda_{r}}+\sqrt{\sum_{i\ge r/2}\lambda_i}}{\sqrt{N}}},
$$
and indeed one may ensure that $\eta_2>c_1\eta_1/\sqrt{m}$ because $m=N/n$.

Turning to the proof of \eqref{ineq:Rademacher_complexity_term_eta_3_final}, recall that $\kappa_4$ is the constant from Lemma \ref{lem:construction_E_Eperp} for which \eqref{ineq:construction_E_Eperp} holds. By the Cauchy-Schwarz inequality and the $L_4-L_2$ norm equivalence,
\begin{equation}
\begin{split}
\label{ineq:Rademacher_complexity_term_eta_3_part1}
&\mathbb{E}\sup_{\theta_1-\theta_2\in (B_{E^{\perp}}-B_{E^{\perp}})}\left|\sum_{i=1}^N \varepsilon_i (\langle X_i,z\rangle\langle X_i,\theta_{1}-\theta_{2}\rangle - \mathbb{E}\langle X_i,z\rangle\langle X_i,\theta_{1}-\theta_{2}\rangle )\right|\\
&\le \mathbb{E}\sup_{\theta_1-\theta_2\in (B_{E^{\perp}}-B_{E^{\perp}})}\left(\left|\sum_{i=1}^N \varepsilon_i \langle X_i,z\rangle\langle X_i,\theta_{1}-\theta_{2}\rangle\right|+ \mathbb{E}\langle X,z\rangle\langle X,\theta_{1}-\theta_{2}\rangle\left|\sum_{i=1}^N \varepsilon_i\right|\right)\\
& \le \mathbb{E}\sup_{\theta_1-\theta_2\in (B_{E^{\perp}}-B_{E^{\perp}})}\left|\sum_{i=1}^N \varepsilon_i \langle X_i,z\rangle\langle X_i,\theta_{1}-\theta_{2}\rangle\right| + 2\kappa^2\sqrt{\kappa_4}\sqrt{N}\sigma(z)\sqrt{\lambda_{r}}.
\end{split}
\end{equation}

Next, let $v_1,\ldots,v_d$ be an orthonormal basis of eigenvectors of the covariance matrix $\Sigma$. Invoking the $L_4-L_2$ norm equivalence once again, we have that
\begin{equation*}
\begin{split}
&\mathbb{E}\sup_{\theta_1-\theta_2\in (B_{E^{\perp}}-B_{E^{\perp}})}\left|\sum_{i=1}^N \varepsilon_i\langle X_i,z\rangle\langle X_i,\theta_{1}-\theta_{2}\rangle\right| \\
&\leq 2 \mathbb{E}\sup_{\theta \in B_{E^{\perp}}}\left|\left\langle \sum_{i=1}^N \varepsilon_i\langle X_i,z\rangle P_{E^{\perp}}X_i,\theta\right\rangle\right|\\
& \leq 2\mathbb{E}\|\sum_{i=1}^N \varepsilon_i\langle X_i,z\rangle P_{E^{\perp}}X_i\|_{\ell_2}\\
& \leq 2 (\mathbb{E}\|\sum_{i=1}^N \varepsilon_i\langle X,z\rangle P_{E^{\perp}}X_i\|_{\ell_2}^2)^{1/2}\\
&= 2\sqrt{N}(\mathbb{E}\langle X,z\rangle^2\|P_{E^{\perp}}X\|_{\ell_2}^2)^{1/2}\\
&= 2\sqrt{N}\left(\sum_{i=1}^d \mathbb{E}\langle X,z\rangle^2\langle P_{E^{\perp}}X,v_i\rangle^2\right)^{1/2} \\
&\le 2\kappa^2\sqrt{N}(\mathbb{E}\langle X,z\rangle^2)^{1/2}\left(\sum_{i=1}^d \mathbb{E}\langle P_{E^{\perp}}X,v_i\rangle^2\right)^{1/2}.
\end{split}
\end{equation*}
Invoking \eqref{ineq:trace_Eperp},  
\begin{equation*}
\mathbb{E}\sup_{\theta_1-\theta_2\in (B_{E^{\perp}}-B_{E^{\perp}})}\left|\sum_{i=1}^N \varepsilon_i\langle X_i,z\rangle\langle X_i,\theta_{1}-\theta_{2}\rangle\right| \leq 2\sqrt{2\kappa_4} \kappa^2 \sqrt{N}\sigma(z) \sqrt{\sum_{i\ge r/2}\lambda_i},
\end{equation*}
and \eqref{ineq:Rademacher_complexity_term_eta_3_final} holds for $C_{2}:=\sqrt{\kappa_4}(2\sqrt{2}+2)$.

Hence, it follows from Theorem \ref{thm:uniform_MOM} that there is a function $\widehat{\nu}(u):S^{d-1}\times (\mathbb{R}^d)^N\rightarrow \mathbb{R}_{+}$ satisfying that with probability at least $1-e^{-cn}$,
\begin{equation*}
\begin{split}
&|\widehat{\nu}(u)-\mathbb{E}\langle X,z\rangle\langle X,P_{E^{\perp}}u\rangle|\le\eta_0+\eta_2\\
&\leq C_{0}(\kappa)\sigma(z) \left(\sqrt{\lambda_{r}}\sqrt{\frac{n}{N}} + \sqrt{\frac{\sum_{i\ge r/2}\lambda_i}{N}}\right)\\
&\le C_0'(\kappa_0,\kappa)\sigma(z)\sqrt{\frac{\sum_{i\ge r/2}\lambda_i}{N}}.
\end{split}
\end{equation*}

Recall that $2^{s_0}=2^{\Tilde{\kappa}_0}n$ by \eqref{def:choice_s0_s1}, and set $\Tilde{\kappa}_0$ to be an absolute constant for which $\log |T_{s_0}|\le cn/2$. Since $z=\pi_{s_0}P_{E}u$, it follows from the union bound over $T_{s_0}$ that with probability at least $$1-e^{-cn}|T_{s_0}|\ge 1-e^{-cn/2},$$ for every $z\in T_{s_0}$
\begin{equation}
\label{ineq:end_termII_part1}
\begin{split}
&|\widehat{\nu}(u)-\mathbb{E}\langle X,\pi_{s_0}P_{E}u\rangle\langle X,P_{E^{\perp}}u\rangle|\\
&\le C_0'\sigma(\pi_{s_0}P_{E}u)\sqrt{\frac{\sum_{i\ge r/2}\lambda_i}{N}}.
\end{split}
\end{equation}
Next, by \eqref{ineq:application_talagrand_MM} and the choice of $s_0$,
\begin{equation}
\label{ineq:end_termII_part2}
\begin{split}
&\sigma(\pi_{s_0}P_{E}u)\le \|\pi_{s_0}P_{E}u-P_{E}u\|_{L_2} + \sigma(P_{E}u)\le \sum_{s\ge s_0+1}\|\Delta_{s}P_{E}u\|_{L_2} + \sigma(P_{E}u)\\
&\le 2^{-(s_0+1)/2}\sum_{s\ge s_0+1} 2^{s/2}\|\Delta_{s}P_{E}u\|_{L_2} + \sigma(P_{E}u)\\
&\le 2^{-(\Tilde{\kappa}_0+1)/2}\kappa_8 \sqrt{\lambda_{r}} + \sigma(P_{E}u).
\end{split}
\end{equation}
Therefore with probability at least $1-e^{-cn/2}$,
\begin{equation}
\label{ineq:end_termII_part3}
\begin{split}
&|\widehat{\nu}(u)-\mathbb{E}\langle X,\pi_{s_0}P_{E}u\rangle\langle X,P_{E^{\perp}}u\rangle|\\
&\leq C_3(\kappa,\kappa_0) \max\{\sigma(P_{E}u),\sqrt{\lambda_{r}}\}\sqrt{\frac{\sum_{i\ge r/2}\lambda_i}{N}}.
\end{split}
\end{equation}
All that is left is to show that 
\begin{equation*}
\sigma(P_{E}u) \leq \max\{\sigma(u),2\sqrt{\kappa_4}\sqrt{\lambda_{r}}\}.
\end{equation*}
Indeed, we clearly have that
\begin{equation*}
|\sigma(P_{E}u) - \sigma(P_{E^{\perp}}u)|^2\le \sigma^2(u)\le 2(\sigma^2(P_{E}u) + \sigma^2(P_{E^{\perp}}u)),
\end{equation*}
and therefore, either $\sigma(P_E u) \geq 2\sigma(P_{E^{\perp}}u)$ (and in particular $\sigma(P_Eu)$ and $\sigma(u)$ are equivalent), or
\begin{equation}
\label{ineq:end_termII_part4}
\sigma(P_Eu)< 2\sigma(P_{E^{\perp}}u)\leq 2\sqrt{\kappa_4}\sqrt{\lambda_{r}}.
\end{equation}
Combining that with \eqref{ineq:end_termII_part3}, there is a constant $\kappa'(\kappa,\kappa_0)$ for which with probability at least $1-e^{-cn/2}$,
\begin{equation}
\label{ineq:final_estimate_termII_chaining}
|\widehat{\nu}(u)-\mathbb{E}\langle X,\pi_{s_0}P_{E}u\rangle\langle X,P_{E^{\perp}}u\rangle|\leq \kappa' \max\{\sigma(u),\sqrt{\lambda_{r}}\}\sqrt{\frac{\sum_{i\ge r/2}\lambda_i}{N}}.
\end{equation}
\medskip

\subsection{Estimating the terms $(I)$ and $(III)$:}
Fix an integer $s\ge s_0+1$ and recall that for every $s$ there are at most $2^{2^s} \cdot 2^{2^{s+1}}\leq 2^{2^{s+2}}$ points of the form $\Delta_s P_E u$. Following the same path as in the previous section, invoking Theorem \ref{thm:uniform_MOM}, there is a function $\widehat{\nu}_s:S^{d-1}\times (\mathbb{R}^{d})^N \rightarrow \mathbb{R}_{+}$ for which for a fixed $z=\Delta_sP_{E}u$, with probability at least $1-2^{-c2^{s}}$, for every $v\in S^{d-1}$,
\begin{equation}
\label{ineq:estimate_each_link}
\begin{split}
&|\widehat{\nu}_s(v)-\mathbb{E}\langle X, z\rangle\langle X,P_{E^{\perp}}v\rangle|\\
&\leq \kappa' \|z\|_{L_2}\left(2^{s/2}\sqrt{\frac{\lambda_{r}}{N}} +  \sqrt{\frac{\sum_{i\ge r}\lambda_i}{N}}\right).
\end{split}
\end{equation}
By the union bound over all elements of the form $\Delta_s P_E u$ and using that $s_0+1 \leq s \leq s_1$, we obtain that the estimate \eqref{ineq:estimate_each_link} holds for every $\Delta_s P_E u$ for $s$ in that range with probability at least
\begin{equation*}
    1-\sum_{s=s_0+1}^{s_1} 2^{2^{s+1}} e^{-C2^{s}} \ge 1- 2e^{-C^\prime2^{s_0}} \geq 1 -e^{-C^{\prime \prime} n}.
\end{equation*}
In particular, \eqref{ineq:estimate_each_link} holds for every $u\in S^{d-1}$ for which $\Delta_s P_{E}u=z$, implying that with probability at least $1-e^{-C'' n}$,
\begin{equation*}
\begin{split}
&|\widehat{\nu}_s(u)-\mathbb{E}\langle X, \Delta_sP_{E}u\rangle\langle X,P_{E^{\perp}}u\rangle|\\
&\leq \kappa' \|\Delta_sP_{E}u\|_{L_2}\left(2^{s/2}\sqrt{\frac{\lambda_{r}}{N}} +  \sqrt{\frac{\sum_{i\ge r}\lambda_i}{N}}\right).
\end{split}
\end{equation*}
On that event, we apply \eqref{lem:application_talagrand_MM} to obtain that for any $u\in S^{d-1}$,
\begin{equation}
\label{rate_TermI}
\begin{split}
&\sum_{s=s_0+1}^{s=s_1}|\widehat{\nu}_s(u)-\mathbb{E}\langle \Delta_sP_{E}u,X\rangle\langle X,P_{E^{\perp}}u\rangle|\\
&\leq \kappa'\left( \sqrt{\frac{\lambda_r}{N}}\sum_{s=s_0+1}^{s_1} 2^{s/2}\|\Delta_sP_Eu\|_{L_2} + \sqrt{\frac{\sum_{i\ge r}\lambda_i}{N}} \sum_{s=s_0+1}^{s_1} \|\Delta_sP_Eu\|_{L_2}\right)
\\
&=\kappa'\left(\sqrt{\frac{\lambda_r}{N}}\sum_{s=s_0+1}^{s_1} 2^{s/2}\|\Delta_sP_Eu\|_{L_2} + \sqrt{\frac{\sum_{i\ge r}\lambda_i}{N}} 2^{-(s_0+1)/2} \sum_{s=s_0+1}^{s_1} 2^{s/2}\|\Delta_sP_Eu\|_{L_2}\right)
\\
&\le \kappa'\sum_{s=s_0+1}^{s_1} 2^{s/2}\|\Delta_sP_Eu\|_{L_2}\left(\sqrt{\frac{ \lambda_{r}}{N}} + \sqrt{\frac{\sum_{i\ge r}\lambda_i}{n N}}\right)\\
&\leq \kappa'\kappa_8 \left(\lambda_{ r} \sqrt{\frac{n}{N}} +  \sqrt{\lambda_{ r}}\sqrt{\frac{\sum_{i\ge r}\lambda_i}{N}}\right)\\
&\leq C_1(\kappa_0,\kappa)\sqrt{\lambda_{ r}}\sqrt{\frac{\sum_{i\ge r/2}\lambda_i}{N}},
\end{split}
\end{equation}
where the last inequality holds because $r=\kappa_En$ implies that $\lambda_rn \le \sum_{i\ge r/2}\lambda_i$, and the wanted estimate on (I) follows.

\medskip

Turning to $(III)$, recall that by Lemma \ref{lem:construction_E_Eperp}, for every $u\in S^{d-1}$, we have that $\mathbb{E}\langle X,P_{E^{\perp}}u\rangle^2\le \kappa_4\lambda_{r}$. Also, it follows from \eqref{def:choice_s0_s1} that $2^{-s_1/2} \le N^{-1/2}$. Thus, we apply \eqref{eq_TalagrandMM_Appendix-0} to obtain that
\begin{equation}
\label{rate_termIII}
\begin{split}
&\mathbb{E}\langle X,P_E u -\pi_{s_1}P_Eu\rangle\langle X,P_{E^{\perp}}u\rangle\le \sqrt{\kappa_4}\sqrt{\lambda_{r}} \|P_E u -\pi_{s_1}P_Eu\|_{L_2}\\
&\le \sqrt{\kappa_4}\sqrt{\lambda_{ r}}\sum_{s\ge s_1}\|\Delta_{s}P_Eu\|_{L_2} \\
&\le\sqrt{\kappa_4} \sqrt{\lambda_{ r}}2^{-s_1/2}\sum_{s\ge s_1}2^{s/2}\|\Delta_{s}P_Eu\|_{L_2}\\
&\leq \kappa_{8} \sqrt{\kappa_4} \lambda_{ r}\sqrt{\frac{n}{N}}\\
&\leq C_2(\kappa,\kappa_0)\sqrt{\lambda_{ r}}\sqrt{\frac{\sum_{i\ge r/2}\lambda_i}{N}}.
\end{split}
\end{equation}
The proof follows by setting $\kappa_{7}$ to be the dominating constant in \eqref{ineq:final_estimate_termII_chaining}, \eqref{rate_TermI} and \eqref{rate_termIII}, namely
\begin{equation*}
\kappa_{7} := 3\max\left\{\kappa', C_1, C_2 \right\}.
\end{equation*}

\subsection{Proof of Lemma \ref{lem:application_talagrand_MM}}
Let $g$ be an isotropic centred gaussian vector and set $G=\Sigma^{1/2} g \sim N(0,\Sigma)$, $G_t:=\langle G,t\rangle$, and $S:=\mathbb{E}\sup_{t\in B_{E}}\langle G,t\rangle$, where $B_{E}$ is the Euclidean unit ball in $E$.

Let $\delta>0$, set $W=\Sigma^{-1/2}\delta B_E\cap B_E $ and observe that
\begin{equation*}
\begin{split}
& \mathbb{E}\sup_{\|t-s\|_{L_2}\le \delta; \ t-s\in 2B_E}|G_t-G_s|=\mathbb{E}\sup_{\|t-s\|_{L_2}\le \delta; \ t-s\in 2B_E}|\langle G,t-s\rangle|\\
&\leq 2\mathbb{E}\sup_{v\in \Sigma^{-1/2}\delta B_E\cap B_E }\langle \Sigma^{1/2}g,v\rangle\\
&=2\mathbb{E}\sup_{v\in W}\langle g,\Sigma^{1/2}v\rangle.
\end{split}
\end{equation*}
Since $\Sigma^{1/2}W = \delta B_{E} \cap \Sigma^{1/2}B_E \subset \delta B_{E} $ and the dimension of $E$ is $r$, it is evident that
\begin{equation*}
\begin{split}
\mathbb{E}\sup_{\|t-s\|_{L_2}\le \delta; \ t-s\in 2B_E}|\langle G,t-s\rangle|&\le 2\mathbb{E}\sup_{u\in \delta B_E}\langle g,u\rangle \\
&\le 2\delta \sqrt{r}.
\end{split}
\end{equation*}

\noindent Let $\delta:= \sqrt{\lambda_{ r}}$ and put $\delta_{k^{\ast}}=\delta$, where $k^{\ast}$ is defined to be the largest integer for which
\begin{equation*}
    2^{-k^{\ast}} S\ge 2\delta \sqrt{ r},
\end{equation*}
thus the first condition of Theorem \ref{thm_TalagrandMM_Appendix} is fulfilled for $k=k^{\ast}$ and $\delta_{k^{\ast}} = \sqrt{\lambda_{r}}$.

As for the second condition, it suffices to verify that
\begin{equation}
\label{ineq:second_conditon_TalagrandMM_Appendix}
    L S2^{-k^{\ast}} 2^{-s_{k^{\ast}}/2} \le \delta_{k^{\ast}} = \sqrt{\lambda_{r}},
\end{equation}
for some well-chosen constant $L>0$. By the maximimality of $k^{\ast}$, we have that 
\begin{equation*}
S2^{-k^{\ast}}\leq 4\sqrt{\lambda_{r}}\sqrt{ r} 
\end{equation*}
and thanks to the choice $s_0$, 
\begin{equation*}
2^{-(s_0+1)/2}= 2^{-(\Tilde{\kappa}_0+1)/2} n^{-1/2}.
\end{equation*}
Recalling that $r=\kappa_En$, setting  $s_{k^{\ast}}= s_0+1$ and $$L:=\frac{2^{(\Tilde{\kappa}_0+1)/2}}{4\sqrt{\kappa_E}},$$
it follows that \eqref{ineq:second_conditon_TalagrandMM_Appendix} holds. 

Therefore, by Theorem \ref{thm_TalagrandMM_Appendix} there is an admissible sequence $(T_{s})$ of $B_{E}$ that satisfies
\begin{equation*}
    \sup_{t\in B_{E}}\sum_{s\ge s_0+1}2^{s/2}\Delta(T_s(t)) \leq L_1 2^{-k^{\ast}} S \leq L_1 \sqrt{ r}\sqrt{\lambda_{r}},
\end{equation*}
for a constant $L_1$ that depends only on $L$, as claimed. \qed

\newpage

\section*{Appendix}
In this appendix, we prove Corollary \ref{cor:lower_bound_gaussianwidth_general_norm}.
\begin{proof}
Let $\gamma_\mu$ be the gaussian measure $\sim \mathcal{N}(0,N^{-1}\Sigma)$. Setting $T=B_2^d$ and following the argument used in the proof of Theorem \ref{lemma:lower_bound_gaussianwidth_general_norm} (see \cite[Page 13 and 14]{depersin2022optimal}), we have that, if for a fixed $\delta<1/4$,

\begin{equation}
\label{ineq:corollary_lower_bound_gaussianwidth}
  \int_{\mu \in \mathbb{R}^d} \mathbb{P}(\|\widehat{\mu}(G_1,\ldots,G_N,\delta)-\mu\|_2\ge r^{\ast}) d_{\gamma_{\mu}}\le \frac{1}{4},
\end{equation}
then
\begin{equation*}
    r^{\ast}\ge \kappa_2\left(\frac{1}{\sqrt{2N}}\ell^{\ast}(\Sigma^{1/2}B_2^d)\right).
\end{equation*}
Thus, it is enough to show that \eqref{ineq:corollary_lower_bound_gaussianwidth} holds for some well-chosen $\delta$. To that end, if $\mu$ is distributed according to $\gamma_{\mu}$, then by the gaussian Lipschitz concentration inequality \cite{ledoux1991probability} and since
$$
\ell^{\ast}(\Sigma^{1/2}B_2^d)\le \sqrt{\tr(\Sigma)},
$$
it follows that
\begin{equation*}
    \mathbb{P}\left(\|\mu\|_2\ge 3\sqrt{\frac{\tr(\Sigma)}{N}}\right)\le e^{-2\tr(\Sigma)/\|\Sigma\|_{2 \to 2}}\le e^{-2}.
\end{equation*}
Setting $\delta = e^{-2}/2$ (which is smaller than $1/4$) and by \eqref{cor:lower_bound_gaussianwidth_general_norm_assumption},
\begin{equation*}
  \int_{\mu \in RB_2^d} \mathbb{P}(\|\widehat{\mu}-\mu\|_2\ge r^{\ast}) d_{\gamma_{\mu}}\le \sup_{\mu \in RB_2^d}\mathbb{P}(\|\widehat{\mu}-\mu\|_2\le r^{\ast}) \le \delta.
\end{equation*}
With this choice of $\delta$,
\begin{equation*}
\int_{\mu \in \mathbb{R}^d} \mathbb{P}(\|\widehat{\mu}-\mu\|_2\ge r^{\ast}) d_{\gamma_{\mu}} \le \delta + \mathbb{P}\left(\|\mu\|_2\ge 3\sqrt{\frac{\tr(\Sigma)}{N}}\right)\le \frac{3}{2}e^{-2}\le \frac{1}{4},
\end{equation*}
as required.
\end{proof}

\end{document}